\documentclass[a4paper]{amsart}

\usepackage{graphicx,amssymb,color,amscd}

\usepackage[all]{xy}
\usepackage{accents}
\SelectTips{cm}{}

\theoremstyle{plain}
\newtheorem{theorem}{Theorem}[section]

\newtheorem*{theorem*}{Theorem}
\newtheorem{lemma}[theorem]{Lemma}
\newtheorem{proposition}[theorem]{Proposition}
\newtheorem{corollary}[theorem]{Corollary}
\theoremstyle{definition}

\theoremstyle{remark}
 
\newtheorem{remark}[theorem]{Remark}

\numberwithin{equation}{section}
\numberwithin{figure}{section}

\newcommand{\bd}{\begin{description}}   
\newcommand{\ed}{\end{description}} 
\newcommand{\ba}{\begin{array}}      \newcommand{\ea}{\end{array}} 
\newcommand{\bc}{\begin{center}}     \newcommand{\ec}{\end{center}} 
\newcommand{\be}{\begin{enumerate}}  \newcommand{\ee}{\end{enumerate}} 
\newcommand{\beq}{\begin{eqnarray}}  \newcommand{\eeq}{\end{eqnarray}} 
\newcommand{\beQ}{\begin{eqnarray*}} \newcommand{\eeQ}{\end{eqnarray*}} 
\newcommand{\bi}{\begin{itemize}}    \newcommand{\ei}{\end{itemize}}

\newcommand\op{\mathrm{op}}
\def\co{\colon\thinspace}
\def\id{\mathop{\mathrm{id}}\nolimits}

\begin{document} 
\title{The universal quantum invariant and colored ideal triangulations} 
\author[S. Suzuki]{Sakie Suzuki} 
\address{Department of Mathematical and Computing Science, School of Computing, Tokyo Institute of Technology, 
2 Chome-12-1 Ookayama, Meguro, Tokyo 152-8552, Japan.}
         \email{sakie@c.titech.ac.jp }
\date{May 1, 2018}
\maketitle

\begin{abstract} 
The Drinfeld double of a finite dimensional Hopf algebra is a quasi-triangular Hopf algebra with the canonical element as the universal $R$-matrix, and one can obtain a ribbon Hopf algebra by adding the ribbon element. The universal quantum  invariant of framed links  is constructed  using a ribbon Hopf algebra. In that construction, a copy of  the universal $R$-matrix is attached to each crossing, and invariance under the Reidemeister III move is shown by the quantum Yang-Baxter equation of the universal $R$-matrix. 
On the other hand,  the Heisenberg double of a finite dimensional Hopf algebra has the canonical element (the \textit{$S$-tensor}) satisfying the pentagon relation. In this paper we reconstruct the universal quantum invariant using the Heisenberg double, and  extend it to an invariant of equivalence classes of  \textit{colored ideal triangulations} of $3$-manifolds up to \textit{colored moves}.  In this construction, a copy of the $S$-tensor is attached to each tetrahedron, and invariance under the \textit{colored Pachner $(2,3)$ moves} is shown by the pentagon relation of the $S$-tensor.
\end{abstract} 

\tableofcontents

\section{Introduction}

The universal quantum invariant \cite{Law89, Law90,O} associated to a ribbon Hopf algebra is an invariant of framed tangles in a cube which has the universal property over Reshetikhin-Turaev invariants \cite{RT90}.   The relationship between the universal quantum invariant and $3$-dimensional, global, topological properties of tangles is not well understood, mainly because of the $2$-dimensional definition using link diagrams.  
In this paper, we give a reconstruction  of the universal quantum invariant using \textit{colored ideal triangulations} of tangle complements, 
and give an extension of  the universal quantum invariant  to an invariant of equivalence classes of colored ideal triangulations of $3$-manifolds up to \textit{colored moves}.
We expect that our framework will become a new method to study the quantum invariants in a $3$-dimensional way.

\subsection{Reconstruction and extension of the universal quantum invariant}

In the theory of quantum groups there are two doubles of a finite dimensional Hopf algebra $A$. One is the \textit{Drinfeld  double} $D(A)$ and the other is the \textit{Heisenberg double} $H(A)$. They are both isomorphic to $A^*\otimes A$ as vector spaces.

The Drinfeld double  $D(A)$ is a quasi-triangular Hopf algebra with a canonical element $R\in D(A)^{\otimes 2}$ as the universal $R$-matrix, which satisfies the quantum Yang-Baxter equation
\begin{align*}
R_{12}R_{13}R_{23}=R_{23}R_{13}R_{12},
\end{align*}
see e.g. \cite{Dri87, Maj98, Maj99}. One can obtain a ribbon Hopf algebra $D(A)^{\theta}$ by adding the ribbon element $\theta$. In what follows we assume that the universal quantum invariant is associated to $D(A)^{\theta}$ for a finite dimensional Hopf algebra $A$.

The Heisenberg double $H(A)$  is a generalization of the Heisenberg algebras \cite{Sem92,Lu94,Kap98}.
Baaj-Skandalis \cite{BS93} and Kashaev \cite{Kas97} showed that a canonical element $S\in H(A)^{\otimes 2}$, which we call the \textit{$S$-tensor}, satisfies the pentagon relation 
\begin{align*}
S_{12}S_{13}S_{23}=S_{23}S_{12}.
\end{align*}
Kashaev \cite{Kas97} also constructed an algebra embedding $\phi\co D(A) \to H(A)\otimes H(A)^{\op}$ such that the image of the universal $R$-matrix is a product of four variants of the $S$-tensor:
\begin{align}\label{rs}
\phi^{\otimes 2}(R)=S_{14}''S_{13}\tilde S_{24} S_{23}' \quad \in \left(H(A)\otimes H(A)^{\op}\right)^{\otimes 2},
\end{align}
where $S', S''$ and $\tilde S$ are the images of $S$ by maps constructed from the antipode, see Theorem \ref{RvsSth}.

The situation (\ref{rs}) reminds us the situation of an  \textit{octahedral triangulation} \cite{CKK14, Yok11, We05}  of the complement of a link in $S^3\setminus \{\pm \infty\}$,  where an octahedron consisting of four tetrahedra is associated to each crossing of a link diagram. \footnote{Throughout this paper we consider only topological ideal triangulations and we do not consider geometric structures on them.}
Actually, corresponding to the formula (\ref{rs}), Kashaev \cite{Kas95} constructed the $R$-matrix consisting of four \textit{quantum dilogarithms} defined in \cite{FK}, and gave a link invariant.  Baseilhac and Benedetti \cite{BB4} also constructed the $R$-matrix consisting of four quantum dilogarithms, each of which is associated to tetrahedron in a singular triangulation of a $3$-manifold, and they recovered Kashaev's $R$-matrix. 
Hikami and Inoue  \cite{HI14, HI15} constructed the $R$-matrix consisting of four \textit{mutations} in a cluster algebra.
Here a mutation is associated to a flip of triangulated surface, where a flip is obtained by attaching a tetrahedron to the surface. They also recovered Kashaev's $R$-matrix  up to a gauge-transformation. 

In this context, it is natural to ask if we can reconstruct the universal quantum invariant of a tangle using an octahedral triangulation of its complement,  where a copy of the $S$-tensor is associated to each tetrahedron in the octahedral triangulation.

The answer is yes, and in this paper we give such a reconstruction. 
Here, we would like to stress that, we can construct the universal quantum invariant using the $S$-tensor by simply rewriting the universal $R$-matrix by (variants of) the $S$-tensor using $\phi^{\otimes 2}$. However, an important result is that we give a way to relate a copy of the $S$-tensor to an ideal tetrahedron in an octahedral triangulation, and a way to read these copies of the $S$-tensor to obtain the universal quantum invariant.
The framework of the above reconstruction enables us to extend the universal quantum invariant to an invariant for \textit{colored singular triangulations}
 of $3$-manifolds up to \textit{colored moves}. 
\footnote{The universal quantum invariant of a tangle is an isotopy invariant, while the extended universal quantum invariant of the complement of the tangle is \textit{not} a topological invariant. That is because there is a \textit{canonical} coloring for the complement of a tangle and the universal quantum invariant of a tangle is equal to the extended universal quantum invariant of its complement with the canonical coloring. }

\subsection{Universal quantum invariant as a state sum invariant with weights in a non-commutative ring}
Let us explain the nature of the coloring on a singular triangulation from a viewpoint of state sum constructions.

One can obtain a state sum invariant of tangles and $3$-manifolds by associating a \textit{$6$j-symbol} to each tetrahedron in a triangulation of a $3$-manifold, where the values of the $6$j-symbol on colors on the edges of a tetrahedron give a weight of the state sum \cite{TV, Oc}.

In the context of hyperbolic geometry, there are several attempts to construct a state sum invariant of hyperbolic links and hyperbolic $3$-manifolds, such that, to each tetrahedron one associates Faddeev and Kashaev's quantum dilogarithm, and the values of them on the cross ratio moduli of hyperbolic ideal triangulation give weights of the state sum.  The first relation between quantum state sums and hyperbolic geometry seems to be  \cite{Kas94} by Kashaev. For an odd integer $N>1$, he proposed a state sum for triangulations of pairs $(M,L)$ of a closed oriented $3$-manifold $M$ and a link $L$ in $M$, using the cyclic $6$j-symbol $R(p,q,r)$ of the Borel subalgebra of $U_q(sl_2)$.
He also showed  that  $R(p,q,r)$ is obtained from certain operators $S$ and $\Psi_{p,q,r}$  on $\mathbb{C}^N\otimes \mathbb{C}^N$, where  $S$ satisfies a certain pentagon relation and  $\Psi_{p,q,r}$ satisfies a version of the quantum dilogarithm identity. A semi-classical
limit of this identity gives
Rogers's identity for
Euler's dilogarithm, and this fact seems to lead Kashaev to his famous conjecture
about the relationship between his invariant and the hyperbolic volumes of link
complements \cite{Kas97'}.

Murakami and Murakami \cite{MM} showed that Kashaev's $R$-matrix is conjugate (up to scalar multiplication) to that of the \textit{colored Jones polynomial} $J_N$ with $q=\exp \frac{2\pi i}{N}$ and an $N$-dimensional irreducible representation of $U_q(sl_2)$. This result
also showed that, in the case of links in the three-sphere, the Kashaev state sums lead to well-defined invariants. 
Murakami-Murakami's construction could be seen as a state sum invariant with a weight associated to a crossing, consisting of four quantum dilogarithms. 

Baseilhac and Benedetti  \cite{BB1,BB2, BB3, BB4} constructed \textit{quantum hyperbolic invariants} 
(QHI)  for triples $(M, L, r)$, where $M$ is a compact oriented $3$-manifold, $L$ is a non-empty link  in $M$, and $r$ is a flat principal bundle over $M$ with structure group $PSL(2,\mathbb{C})$.
These invariants are obtained by adapting and generalizing the constructions of Kashaev, and in the case where $M$ is the three-sphere and $r$ is the trivial flat bundle,  they recovered the Kashaev invariants. 
In \cite{BB5}, they reorganized QHI as invariants for tuples $(M, L, r, \kappa)$, where $\kappa$ is a family of cohomological classes called weights. In this version, the QHI are defined by state sums where tensors called matrix dilogarithms (related to the cyclic $6$j-symbols $R(p,q,r)$) 
 are associated to tetrahedra in a singular triangulation. 
The arguments of the matrix dilogarithms are certain special systems of $N$th
roots of hyperbolic shape parameters on the tetrahedra, encoding the flat bundle $r$ and the weights $\kappa$.

On the other hand, the universal quantum invariant could be seen as a state sum invariant with weights being tensors of a ribbon Hopf algebra; a weight is associated to each fundamental tangle (see Figure \ref{fig:fundamental}), especially a copy of the universal $R$-matrix is associated to each crossing, and one takes products of the weights in the order following the orientations of strands of a tangle (see Section \ref{univinv} for the precise definition). We would like to apply this framework to a state sum construction using  triangulations, i.e., our motto (framework)  is:

\begin{center}
\textit{Using an element $S$ satisfying a pentagon relation in an (non-commutative) algebra, 
construct a state sum invariant of $3$-manifolds by associating  a copy of $S$  to each tetrahedron of a (singular) triangulation.} 
\end{center}
The state sum invariants using $6$j-symbols (resp. quantum dilogarithms) could be treated in this framework as functions, rather than as its values in $\mathbb{C}$,  on colors on edges of tetrahedra (resp. cross ratio moduli of ideal tetrahedra \cite{BB5})
and we expect to obtain those invariants from the universal quantum invariant  naturally keeping this framework.

In the above framework one does not need to fix colors on the edges of a tetrahedron or cross ratio modulus of an ideal tetrahedron, and for the proof of invariance of state sums, instead of the pentagon identity of  $6$j-symbols or of quantum dilogarithms, one would work with an algebraic pentagon relation. 
Moreover, we expect that such an invariant involves combinatorial information of a triangulation in its non-commutative algebra structure, including the consistency and the completeness conditions of ideal triangulations when we fix cross ratio moduli. 

When we use  a (singular) triangulation, we do not have a canonical order on the set of weights on tetrahedra in the triangulation. Thus we need to fix an order, then we naturally come to a notion of the colored singular triangulation 
\footnote{The notion of  colored singular triangulations can be interpreted by the notion of branchings \cite{BS, BP, BB5, BB6}, see Remark \ref{branch}. In this paper we keep the former one since it is defined combinatorially and fit to our purpose. When one would like to see geometric properties of state sums, then the latter one would make more sense.}
: each tetrahedron is sticked by two strands and strands are connected  globally in the triangulation. Then a copy of the $S$-tensor is associated to the two strands of each tetrahedron and we can read the copies of the $S$-tensor in the order following the orientations of the  strands.
Corresponding to the \textit{Pachner $(2,3)$ move} and the \textit{$(0,2)$ move} of singular triangulations, 
we define  \textit{colored Pachner $(2,3)$ moves} and \textit{colored  $(0,2)$ moves} of colored singular triangulations.
The extension of the universal quantum invariant is an invariant of colored singular triangulations up to certain \textit{colored moves}.
In this paper these strands first arise from a tangle diagram, and
then we consider strands more generally in singular triangulations of  topological spaces. 

\subsection{Organization of this paper}

Section \ref{univinv0} is devoted to the definition of the universal quantum invariant associated to a ribbon Hopf algebra.
In Section \ref{dh}, we recall the Drinfeld double $D(A)$ and the Heisenberg double $H(A)$ of a finite dimensional Hopf algebra $A$, where the universal $R$-matrix in  $D(A)^{\otimes 2}$ and the $S$-tensor in $H(A)^{\otimes 2}$ satisfy the quantum Yang-Baxter equation and the pentagon equation, respectively. We also recall from \cite{Kas97} how these  elements are related via an embedding of $D(A)$ into $H(A) \otimes H(A)^{\op}$.
In Section \ref{j''},  we  give a reconstruction of the universal quantum invariant using the Heisenberg double.
In Section \ref{exte} we define colored diagrams and extend the universal quantum invariant to an invariant of colored diagrams up to colored moves.
Section \ref{cit} and Section \ref{gotc}  are devoted to $3$-dimensional descriptions of the reconstruction and the extension of the universal quantum invariant.
In Section \ref{cit} we define colored singular  triangulations of topological spaces. The universal quantum invariant can be considered as an invariant of the colored singular  triangulations.
In Section \ref{gotc} we define colored ideal triangulations of tangle complements
\footnote{For links, this construction corresponds to the branched triangulations defined in \cite[Section 2.3]{BB4} without walls. }
 arising from octahedral triangulations, which have been studied in e.g., \cite{CKK14,Yok11} in the context of the hyperbolic geometry.

\subsection{Acknowledgments}
This work was partially supported by JSPS KAKENHI Grant Number 15K17539.
The author is deeply grateful to Kazuo Habiro and Tomotada Ohtsuki
for helpful advice and encouragement. 
She would like also to thank St\'ephane Baseilhac, Riccardo Benedetti, Naoya Enomoto, Stavros Garoufalidis, Rei Inoue, Rinat Kashaev, Akishi Kato, Seonhwa Kim, Thang Le and  Yuji Terashima for their helpful discussions  and   comments.

\section{Universal quantum invariant}\label{univinv0}

In this paper, a \textit{tangle} means a proper embedding in a cube $[0,1]^3$ of a compact, oriented $1$-manifold, whose
boundary points are on the two parallel lines $[0,1] \times \{0,1\}\times \{1/2\}$.
A \textit{tangle diagram} is a diagram of a tangle obtained from the projection $p\co (x,y,z) \mapsto (x,y,0)$ to the $(x,y)$-plane, see Figure \ref{tangle}.
A \emph{framed} tangle is a tangle equipped with a trivialization of its normal tangent bundle, which is presented in a diagram by the blackboard framing.

\begin{figure}
\includegraphics[width=10cm,clip]{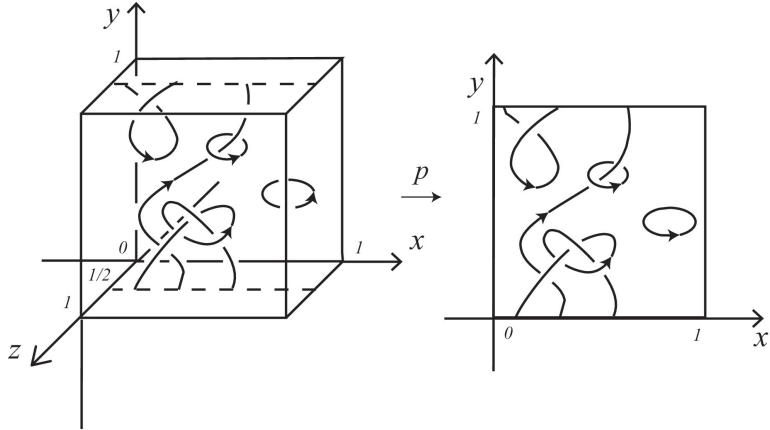}
\caption{A tangle and its diagram}\label{tangle}
\end{figure}
\subsection{Ribbon Hopf algebras}

Let $(A, \eta_A, m_A, \varepsilon_A, \Delta_A, \gamma_A)$ be a finite dimensional Hopf algebra over a field $k$, with $k$-linear maps
\begin{align*}
\eta_A&\co k \to A,
\\
\varepsilon_A&\co A\to k,
\\
m_A&\co A\otimes A \to A,
\\
\Delta_A&\co A \to A\otimes A,
\\
\gamma_A&\co  A \to A,
\end{align*}
which are called \textit{unit},  \textit{counit},  \textit{multiplication},  \textit{comultiplication},  and \textit{antipode}, respectively. 
For simplicity we will omit the subscript $A$ of each map above when there is no confusion.

For distinct integers $1\leq j_1,\ldots, j_m \leq l$  and $x=\sum x_1\otimes \cdots \otimes  x_m\in A^{\otimes m}$, 
we use the notation
\begin{align}\label{subs}
x^{(l)}_{j_1\ldots j_m}=\sum (x_1)_{j_1}\cdots (x_m)_{j_m} \in A^{\otimes l}, 
\end{align}
 where $(x_i)_{j_i}$ represents the element in $A^{\otimes l}$ obtained by placing $x_i$ on the $j_i$th tensorand, i.e.,
$$(x_i)_{j_i}=1\otimes \cdots \otimes x_i
\otimes \cdots \otimes 1,
$$
where $x_i$ is at the ${j_i}$th position.
For example, for $x=\sum x_1\otimes x_2 \otimes x_3,$  we have $x^{(3)}_{312}=\sum x_2 \otimes x_3 \otimes x_1$.
Abusing the notation, we will omit the superscript of $x^{(l)}_{j_1\ldots j_m}$ as $x_{j_1\ldots j_m}$.

For $k$-modules $V, W$, we define the symmetry map
\begin{align}\label{symmetry}
\tau_{V,W}\co  V\otimes W \to W \otimes V, \quad a\otimes b \mapsto b\otimes a.
\end{align}

A quasi-triangular Hopf algebra $(A, \eta, m, \varepsilon, \Delta, \gamma, R)$ is a Hopf algebra $(A, \eta, m, \varepsilon, \Delta, \gamma)$
with an invertible element  $R\in A ^{\otimes 2}$, called \textit{the universal $R$-matrix},  such that

\begin{align*}
&\Delta^{\op}(x) =R \Delta (x) R^{-1} \quad \text{for } x\in A,
\\
&(\Delta\otimes 1)(R)=R_{13}R_{23}, \quad 
(1\otimes \Delta)(R)=R_{13}R_{12},
\end{align*}
where $\Delta^{\op}=\tau_{A,A} \circ \Delta$.

A ribbon Hopf algebra $(A, \eta, m, \varepsilon, \Delta, \gamma, R, \textbf{r})$, see  e.g., \cite{Kas95}, is a quasi-triangular Hopf algebra $(A, \eta, m, \varepsilon, \Delta, \gamma, R)$
with a central, invertible element $\textbf{r} \in A $, called \textit{ribbon element}, such that

\begin{align*}
&\textbf{r}^{2}=u\gamma(u), \quad \gamma(\textbf{r})=\textbf{r},\quad \varepsilon(\textbf{r})=1, \quad \Delta(\textbf{r})=(R_{21}R)^{-1}(\textbf{r}\otimes \textbf{r}),
\end{align*}
where 
\begin{align}\label{u}
u=\sum  S(\beta)\alpha
\end{align}
 with $R=\sum \alpha \otimes \beta$.


\subsection{Universal quantum invariant for framed tangles}\label{univinv}
In this section, we recall the universal quantum invariant  \cite{O, Law89, Law90}  for framed tangles associated to a ribbon Hopf algebra $(A, \eta, m, \varepsilon, \Delta, \gamma, R, \textbf{r})$.

Let $T=T_1\cup \cdots \cup T_n$ be an $n$-component, framed, ordered tangle.

Set $$N=\mathrm{Span}_k\{ab-ba \ |\ a,b \in A\}\subset A.$$

For $i=1,\ldots, n$, let 
\begin{align*}
 A_i=\begin{cases}
A \quad &\text{if} \quad \partial T_i \not=  \emptyset,
\\
A/N \quad &\text{if} \quad \partial T_i=  \emptyset.
\end{cases}
\end{align*}

We define the universal quantum invariant $J(T)\in  A_1 \otimes \cdots \otimes  A_n$  in three steps as follows. We follow the notation in \cite{sakie2}.

\textbf{Step 1. Choose a diagram.}
We choose a diagram $D$ of $T$ which is obtained by pasting, horizontally and vertically, 
copies of the fundamental tangles depicted in Figure \ref{fig:fundamental}.  
\begin{figure}[!h]
\centering
\includegraphics[width=9cm,clip]{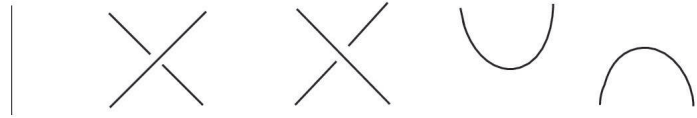}
\caption{Fundamental tangles, where the orientation of each strand is arbitrary
 }\label{fig:fundamental}
\end{figure}

\textbf{Step 2.  Attach labels.}
We attach labels on the copies of the fundamental tangles in the diagram, 
following the rule described in Figure \ref{fig:cross}, where each $\gamma'$ should be replaced with $\gamma$
if the string is oriented upwards, and with the identity otherwise.  
We do not attach any label to the other copies of fundamental tangles, i.e.,  to a straight strand and to a local maximum or minimum oriented from right to left.

\begin{figure}[!h]
\centering
\begin{picture}(300,70)
\put(50,25){\includegraphics[width=7.5cm,clip]{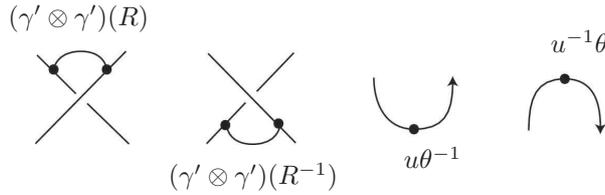}}
\put(40,70){$(\gamma'\otimes \gamma')(R)$}
\put(100,10){$(\gamma'\otimes \gamma')(R^{-1})$}
\put(188,16){$u\theta ^{-1}$}
\put(243,60){$u^{-1}\theta$}
\end{picture}
\caption{How to place labels  on the fundamental tangles}\label{fig:cross}
\end{figure}

\textbf{Step 3.  Read the labels.}
We define the $i$th tensorand of $J({D})$ as the product 
of the labels on the $i$th component of $D$, where the labels are read off along $T_i$
reversing the orientation, and  written from left to  right. Here, if $T_i$ is a closed component, then we choose arbitrary point $p_i$ on $T_i$ and  read the label from $p_i$.  
The labels on the crossings are read as in Figure   \ref{fig:cross2}.

\begin{figure}[!h]
\centering
\begin{picture}(300,120)
\put(50,20){\includegraphics[width=5.8cm,clip]{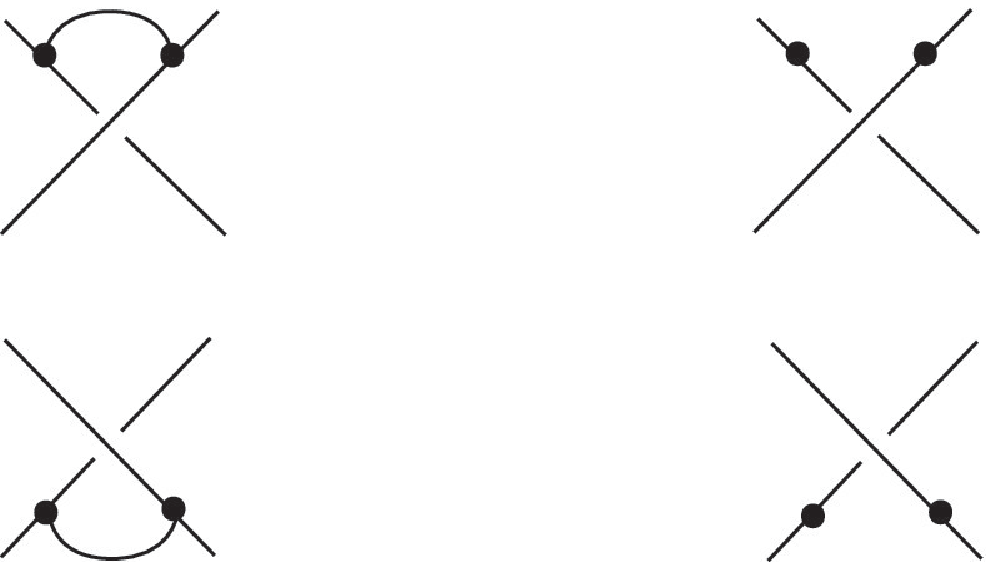}}
\put(42,120){$(\gamma'\otimes \gamma')(R)$}
\put(100,90){$= \quad \sum$}
\put(150,95){$\gamma'(\alpha)$}
\put(208,95){$\gamma'(\beta)$}

\put(42,10){$(\gamma'\otimes \gamma')(R^{-1})$}
\put(100,40){$=\quad \sum $}
\put(150,35){$\gamma'(\alpha^{-})$}
\put(208,35){$\gamma'(\beta^{-})$}
\end{picture}
\caption{How to read the labels on crossings, where $R^{-1}=\sum \alpha^- \otimes \beta ^-$}\label{fig:cross2}
\end{figure}

As is well known \cite{O},  $J(T):=J(D)$ does not depend on the choice of the diagram and the base points $p_i$, and thus defines an isotopy invariant of  tangles.

\begin{figure}[!h]
\centering
\begin{picture}(300,100)
\put(70,20){\includegraphics[width=4.5cm,clip]{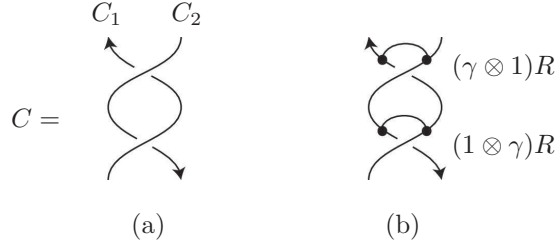}}
\put(65,80){$C_1$}
\put(95,80){$C_2$}
\put(35,40){$C=$}
\put(200,60){$(\gamma \otimes 1)R$}
\put(200,30){$(1\otimes \gamma)R$}
\put(80,0){(a)}
\put(175,0){(b)}
\end{picture}
\caption{(a) A  tangle diagram $C$, (b) The label put on $C$ }\label{fig:T_h}
\end{figure}

For example, for the tangle $C=C_1\cup C_2$ shown in Figure \ref{fig:T_h}, we have
\begin{align}
J({C})&=\sum \gamma(\alpha )\gamma(\beta') \otimes \alpha'\beta,
\end{align} 
where $R=\sum \alpha\otimes \beta=\sum \alpha'\otimes \beta'.$

\section{Drinfeld double and Heisenberg double}\label{dh}

Let $(A, \eta, m, \varepsilon, \Delta, \gamma)$ be a finite dimensional Hopf algebra. 
Let $A^*=\mathrm{Hom}_k(A,k)$. 
Define the pairing
\begin{align}\label{pairex}
\langle \ , \ \rangle \co A^{*} \otimes A \to k,  \quad f\otimes x \mapsto f(x),
\end{align}
and  extend it  to  
$$\langle \ , \ \rangle\co (A^{*})^{\otimes n} \otimes A^{\otimes n} \to k,$$
for $n\geq 1$,  by
\begin{align*}
\langle f_{1}\otimes \cdots \otimes f_{n}, x_{1}\otimes \cdots \otimes x_{n} \rangle=\langle f_{1}, x_{1} \rangle \cdots \langle f_{n}, x_{n} \rangle.
\end{align*}
Then the dual Hopf algebra 
$$A^*=(A^*, \eta_{A^*}=\varepsilon^*, m_{A^*}=\Delta^*, \varepsilon_{A^*}=\eta^*, \Delta_{A^*}=m^*,  \gamma_{A^*}=\gamma^*)$$ is defined using the transposes of the morphisms of $A$, i.e.,  is defined  uniquely by
\begin{align*}
\langle \varepsilon ^*(a), x\rangle&=a\varepsilon (x), \quad a\in k, x\in A,
\\
\langle \Delta^*(f\otimes g), x\rangle&=\langle f\otimes g, \Delta(x)\rangle, \quad f,g \in A^{*}, x\in A,
\\
\eta^*(f)a&=\langle f, \eta (a)\rangle, \quad  f\in A^{*}, a\in k,
\\
\langle m^*(f), x\otimes y\rangle&=\langle f, m(x\otimes y)\rangle, \quad f\in A^{*}, x, y\in A,
\\
\langle \gamma^*(f), x\rangle&=\langle f, \gamma (x) \rangle, \quad  f\in A^{*}, x\in A.
\end{align*}

\subsection{Drinfeld double and Yang-Baxter equation}
For any finite dimensional Hopf algebra with invertible antipode, the Drinfeld quantum double construction gives a quasi-triangular  Hopf algebra \cite{Dri87}.
Here, we follow the notation in \cite{Ka95}.

Let $(A, \eta, m, \varepsilon, \Delta, \gamma, \gamma^{-1})$ be a finite dimensional Hopf algebra with invertible antipode, $A^{\op}=(A, \eta, m^{\op}, \varepsilon, \Delta, \gamma^{-1}, \gamma)$  the opposite Hopf algebra and 
$(A^{\op})^*=(A^*,  \varepsilon^*, \Delta^*, \eta^*, (m^{\op})^*, (\gamma^{-1})^*, \gamma^*)$  the dual of the opposite Hopf algebra, where $m^{\op}=m\circ \tau_{A,A}$.
For simplicity, we set
$$\bar \gamma=\gamma^{-1}.$$ 

Let $\Delta ^{(0)}=\id$ and $\Delta^{(n)}=(\Delta \otimes 1^{\otimes n-1})\Delta ^{(n-1)}$ for $n\geq 1$. 
In what follows, for $x\in A$ or  $x\in A^*$, we use the notation 
\begin{align*}
\Delta(x)=\Delta^{(1)}(x)&=\sum x'\otimes x''=\sum x^{(1)} \otimes x^{(2)},
\\
 (\Delta \otimes 1)\Delta(x)=\Delta^{(2)}(x)&=\sum x'\otimes x''\otimes x'''=\sum x^{(1)} \otimes x^{(2)} \otimes x^{(3)},
\\
\Delta^{(n)}(x)&=\sum x^{(1)}\otimes \cdots \otimes x^{(n+1)},
\end{align*}
for $n \geq 3$.
We have $$(m^{\op})^*(f)=\Delta^{\op}(f)=\sum f''\otimes f'$$ for $f\in (A^{\op})^*$. \footnote{In \cite{Ka95}, he uses the notation $\Delta^{\op}(f)=\sum f'\otimes f''$.}

There is a unique left action 
$$A\otimes (A^{\op})^*  \to (A^{\op})^*, \quad a \otimes f \mapsto a\cdot f,$$ such that
\begin{align*}
\langle a \cdot f, x\rangle= \sum \langle f,  \bar \gamma(a'')xa'\rangle,
\end{align*}
for $a,x\in A$ and $f\in (A^{\op})^*$, which induces the left $A$-module coalgebra structure on $(A^{\op})^*$.
Also, there is a unique right action 
$$A\otimes (A^{\op})^*  \to A, \quad a \otimes f \mapsto a^f,$$  such that
\begin{align*}
a^f = \sum f(\bar \gamma(a''')a')a''
\end{align*}
for $a\in A $
and $f\in (A^{\op})^*$, which induces the right  $(A^{\op})^*$-module coalgebra structure on $A$.

The Drinfeld double $$D(A)=((A^{\op})^*\otimes A, \eta_{D(A)}, m_{D(A)}, \varepsilon_{D(A)}, \Delta_{D(A)}, \gamma_{D(A)},R)$$ is the quasi-triangular Hopf algebra  defined as the bicrossed product of  $A$ and $(A^{\op})^*$.
Its unit, counit, and comultiplication are given by these of $(A^{\op})^*\otimes A$, i.e., we have
\begin{align*}
\eta_{D(A)}(1)&=\eta_{(A^{\op})^*\otimes A}(1)=1\otimes 1,
\\
\varepsilon_{D(A)}(f\otimes a)&=\varepsilon_{(A^{\op})^*\otimes A}(f\otimes a)=f(1)\varepsilon (a),
\\
\Delta_{D(A)}(f\otimes a)&=\Delta_{(A^{\op})^*\otimes A}(f\otimes a)=\sum f''\otimes a'\otimes f'\otimes a'',
\end{align*}
for $a\in A$ and $f\in (A^{\op})^*$.
Its multiplication is given 
 by
\begin{align}\label{mudri}
m_{D(A)}\left((f\otimes a)\otimes (g\otimes b)\right)&=\sum f(a'\cdot g'')\otimes a''^{g'}b
\\
&=\sum fg(\bar \gamma(a''')?a')\otimes a''b,
\end{align}
for $a,b\in A$ and $f,g\in (A^{\op})^*$, where the question mark $?$ denotes
 a place of the variable.
 Its antipode is given by
\begin{align*}
\gamma_{D(A)}(f\otimes a)&=\sum \gamma(a'')\cdot \bar \gamma^*(f')\otimes \gamma(a')^{\bar \gamma^*(f'')},
\end{align*}
for $a\in A$ and $f\in (A^{\op})^*$.

Fix a basis $\{e_{a}\}_{a\in \mathcal{I}}$ of $A$  and its dual basis $\{e^{a}\}_{a\in \mathcal{I}}$ of $A^{*}$.
The universal $R$-matrix is defined as the canonical element 
\begin{align*}
R=\sum_a (1\otimes e_a)\otimes (e^a \otimes 1) \in D(A)\otimes D(A).
\end{align*}

\subsection{Heisenberg double and pentagon relation}

Let $A$ be a finite dimensional  Hopf algebra  with an invertible antipode as in the previous section.
The Heisenberg double $$H(A)=(A^*\otimes A, \eta_{H(A)}, m_{H(A)})$$ is the algebra with the unit
$\eta_{H(A)}(1)=\eta_{A^*\otimes A}(1)=1\otimes 1$ and the multiplication
\begin{align}\label{muhei}
m_{H(A)}\left((f\otimes a)\otimes (g\otimes b)\right)&=\sum fg(?a')\otimes a''b,
\end{align} 
for $a,b\in A$ and $f,g\in (A^{\op})^*$.

Kashaev showed the following.

\begin{theorem}[{Kashaev \cite{Kas97}}]\label{Spr}
The canonical element 
$$S=\sum _a(1\otimes e_{a}) \otimes (e^{a} \otimes 1)\in H(A) \otimes H(A)$$ satisfies the pentagon relation
\begin{align}\label{pentagon}
S_{12}S_{13}S_{23}=S_{23}S_{12} \quad \in H(A)^{\otimes 3}.
\end{align}
\end{theorem}
\subsection{Drinfeld double and Heisenberg double}

Let  $$ H(A^*)=(A\otimes A^*, \eta_{H(A^*)}, m_{H(A^*)})$$ be the Heisenberg double of the dual Hopf algebra $A^*$ of $A$, where we identify $(A^*)^*$ and $A$ in the standard way.

Set $A^{\mathrm{opcop}}=(A, \eta, m^{\op}, \varepsilon, \Delta^{\op}, \gamma, \gamma^{-1})$.
We have the following lemma.
\begin{lemma}\label{comu}
The algebras $H(A^*)$ and $H(A)^{\op}$ are isomorphic via the unique isomorphism $\Gamma\circ \tau$
such that
\begin{align*}
&
 \tau=\tau_{A^*,A}\co  H(A^{*})\to H(A^{\mathrm{opcop}})^{\op}, 
\quad   x\otimes  f \mapsto  f \otimes x,
\\
 &\Gamma= \bar{\gamma}^*\otimes \gamma\co H(A^{\mathrm{opcop}}) ^{\op}
\to H(A)^{\op},\quad  f \otimes  x\mapsto \bar \gamma^*(f)\otimes \gamma (x).
\end{align*}

\end{lemma}
\begin{proof}
We have
\begin{align*}
\tau (x\otimes  f )\cdot_{ H(A^{\mathrm{opcop}})^{\op}} \tau (y\otimes g ) 
=&\left( f \otimes  x\right) \cdot_{ H(A^{\mathrm{opcop}})^{\op}} \left(g \otimes  y\right)
\\
=& \left(g \otimes  y\right)\cdot_{ H(A^{\mathrm{opcop}})}\left( f \otimes  x\right)
\\
=&\sum  g  \cdot_{ (A^*)^{\op}} \langle  f, ? \cdot_{ A^{\op}} y''\rangle \otimes   y'\cdot_{ A^{\op}} x
\\
=&\sum  \langle f, y''?\rangle g\otimes x y'
\\
=&\sum  \langle f', y''\rangle f'' g\otimes x y'
\\
=&\sum   f'' g \otimes  x y' \langle f', y'' \rangle
\\
=&\sum \tau \left(x y' \langle f', y'' \rangle \otimes f''g \right) 
\\
=&\sum \tau \left(x \langle ? f',y \rangle \otimes f''g \right) 
\\
=&\tau \left((x\otimes f )\cdot_{ H(A^{*})}  (y\otimes  g )\right), 
\end{align*}
and we have
\begin{align*}
\Gamma(f \otimes  x)\cdot_{H(A)^{\op}} \Gamma( g \otimes  y)
=&\left(\bar \gamma^*(f)  \otimes \gamma(x)\right) \cdot_{ H(A)^{\op}} \left(\bar \gamma^*(g)  \otimes \gamma(y)\right)
\\
=&\left(\bar \gamma^*(g)  \otimes \gamma(y)\right)\cdot_{H(A)}\left(\bar \gamma^*(f)  \otimes \gamma(x)\right) 
\\
=&\sum\bar \gamma^*(g) \langle \bar \gamma ^*(f), ?\gamma(y)'\rangle \otimes \gamma(y)''\gamma(x)
\\
=&\sum\bar \gamma ^* (g)\bar \gamma ^*(f)' \langle \bar \gamma ^* (f)'',\gamma (y)' \rangle \otimes  \gamma(y)''\gamma(x)
\\
=&\sum\bar \gamma ^* (g)\bar \gamma ^*(f'') \langle \bar \gamma ^* (f'),\gamma (y'') \rangle \otimes  \gamma(y')\gamma(x)
\\
=&\sum (\bar \gamma ^* \otimes \gamma)\left( \langle f',y'' \rangle f''g \otimes x y' \right)
\\
=&\sum (\bar \gamma ^* \otimes \gamma)\left( \langle f,y''? \rangle g \otimes x y' \right)
\\
=&\sum \Gamma\left( g \cdot_{ (A^*)^{\op}} \langle  f, ?\cdot_{ A^{\op}}  y''\rangle  \otimes   y'  \cdot_{ A^{\op}} x\right)
\\
=&\Gamma\left(( g \otimes  y) \cdot_{H(A^{\mathrm{opcop}})}( f \otimes  x) \right)
\\
=&\Gamma\left(( f \otimes  x)\cdot_{H(A^{\mathrm{opcop}})^{\op}} (g \otimes  y) \right),
\end{align*}
which completes the proof.
\end{proof}

Set
\begin{align}
\phi(1\otimes e_{a})= & \sum 1\otimes e_a' \otimes 1 \otimes \gamma (e_a'')\label{ext}
\in H(A)\otimes H(A)^{\op},
\\
\phi(e^{a}\otimes 1)=&\sum (e^a)'' \otimes 1 \otimes \bar \gamma^* ((e^a)')\otimes 1\label{ext2}
 \in H(A)\otimes H(A)^{\op}.
\end{align}

Kashaev \cite{Kas97} stated without proof that the Drinfeld double $D(A)$ can be realized as a subalgebra in the tensor product $H(A)\otimes H(A)^{\op}$
of the Heisenberg double $H(A)$ and its opposite algebra $H(A)^{\op}$ as follows. \footnote{In \cite{Kas97} he uses $H(A^*)$ instead of $H(A)^{\op}.$}

\begin{theorem}[Kashaev \cite{Kas97}]\label{phi}
There is a unique algebra homomorphism 
\begin{align}
\phi \co D(A) \to H(A)\otimes H(A)^{\op}
\end{align} extending (\ref{ext}) and (\ref{ext2}).
  \end{theorem}
We give the proof of Theorem \ref{phi} by showing $\phi$ explicitly.

\begin{proof}[{Proof of Theorem \ref{phi}}]
We define $\phi \co D(A) \to H(A)\otimes H(A)^{\op}$ by
\begin{align*}
\phi &= m_{H(A)\otimes H(A)^{\op}}\circ ((1\otimes \eta)^{\otimes 2} \otimes  (\eta \otimes 1)^{\otimes 2})\circ(1 \otimes  \bar \gamma^* \otimes 1 \otimes \gamma)\circ(\Delta^{\op}\otimes \Delta),
\end{align*}
i.e., we have
\begin{align*}
\phi(f\otimes x) &=\sum \langle \bar \gamma^* \left(f' \right)'', \gamma (x'') ' \rangle f''\otimes x'\otimes\bar \gamma^* \left(f' \right)'\otimes \gamma (x'')''
\\
&=\sum \langle  f', x''' \rangle f'''\otimes x'\otimes  \bar \gamma^* (f'' )\otimes \gamma (x''),
\end{align*}
for $f\in A^*$ and $x\in A$.

The map $\phi$ is an algebra homomorphism as follows.

\begin{align*}
\phi(1\otimes x)\phi(1\otimes y)&=\left( \sum  1\otimes x'\otimes 1 \otimes \gamma (x'')\right)
\cdot_{H(A)\otimes H(A)^{\op}}\left(\sum  1\otimes y'\otimes 1 \otimes \gamma (y'')\right)
\\
&= \sum  1\otimes x'y'\otimes 1 \otimes \gamma (x'') \cdot_{A^{\op}} \gamma (y'')
\\
&= \sum  1\otimes (xy)'\otimes 1 \otimes \gamma ((xy)'')
\\
&=\phi\left((1\otimes x) \cdot_{D(A)} (1\otimes y)\right),
\end{align*}

\begin{align*}
\phi(f\otimes 1)\phi(g\otimes 1)&=\left(\sum   f''\otimes 1 \otimes \bar \gamma^* (f')\otimes 1\right)
\cdot_{H(A)\otimes H(A)^{\op}}\left(\sum   g''\otimes 1 \otimes \bar \gamma^* (g')\otimes 1\right)
\\
&=\sum   f''g''\otimes 1 \otimes \bar \gamma^* (f') \cdot_{(A^*)^{\op}} \bar \gamma^* (g')\otimes 1
\\
&=\sum   (fg)''\otimes 1 \otimes \bar \gamma^* ((fg)')\otimes 1
\\
&=\phi\left((f\otimes 1) \cdot_{D(A)} (g\otimes 1)\right),
\end{align*}

\begin{align*}
\phi(f\otimes 1)\phi(1\otimes x)&=\left(\sum   f''\otimes 1 \otimes \bar \gamma^* (f')\otimes 1\right)
\cdot_{H(A)\otimes H(A)^{\op}}\left(\sum  1\otimes x'\otimes 1 \otimes \gamma (x'')\right)
\\
&=\sum f''\otimes x'\otimes \langle \bar \gamma^*(f')'',\gamma(x'')'\rangle  \bar \gamma^*\left( f'\right)'\otimes  \gamma (x'')''
\\
&=\sum\langle f', x'''\rangle f'''\otimes x'\otimes  \bar \gamma^*\left( f''\right)\otimes  \gamma (x'')
\\
&=\phi(f\otimes x)
\\
&=\phi\left( (f\otimes 1)\cdot_{D(A)} (1\otimes x)\right), 
\end{align*}
\begin{align*}
\phi(1\otimes x)\phi(f\otimes 1)&=\left(\sum  1\otimes x^{(1)}\otimes 1 \otimes \gamma (x^{(2)})\right)
\cdot_{H(A)\otimes H(A)^{\op}}\left(\sum   f^{(2)}\otimes 1 \otimes \bar \gamma^* (f^{(1)})\otimes 1\right)
\\
&=\sum\langle f^{(3)}, x^{(1)}\rangle f^{(2)}\otimes x^{(2)}\otimes \bar \gamma^*\left( f^{(1)}\right)\otimes  \gamma (x^{(3)})
\\
&=\sum\langle f^{(4)}, x^{(1)}\rangle\varepsilon (f^{(1)})\varepsilon\left(x^{(4)}\right)
 f^{(3)}\otimes x^{(2)}\otimes  \bar\gamma^*\left( f^{(2)}\right)\otimes  \gamma (x^{(3)})
\\
&=\sum\langle f^{(4)}, x^{(1)}\rangle\langle  f^{(1)}, \bar\gamma( x^{(5)}) x^{(4)}\rangle
 f^{(3)}\otimes x^{(2)}\otimes \bar \gamma^*\left( f^{(2)}\right)\otimes  \gamma (x^{(3)})
\\
&=\sum\langle f^{(5)}, x^{(1)}\rangle\langle  f^{(1)}, \bar\gamma( x^{(5)})\rangle\langle  f^{(2)}, x^{(4)}\rangle
 f^{(4)}\otimes x^{(2)}\otimes \bar \gamma^*\left( f^{(3)}\right)\otimes  \gamma (x^{(3)})
\\
&=\phi\left( \langle f^{(1)},\gamma(x^{(3)})\rangle  \langle f^{(3)}, x^{(1)}\rangle f^{(2)}\otimes x^{(2)}\right)
\\
&=\phi\left( (1\otimes x)\cdot_{D(A)} (f\otimes 1)\right),
\end{align*}
where the fourth identity follows from $m^{\op}(1\otimes \bar \gamma )\Delta=\eta\varepsilon$.

The map $\phi$ satisfies  (\ref{ext}) and (\ref{ext2}) as follows.
\begin{align*}
\phi(1\otimes e_a)
&=\sum \langle  1, e_a''' \rangle 1\otimes e_a'\otimes 1\otimes \gamma \left(e_a'')\right).
\\
&=\sum  1\otimes e_a'\otimes 1 \otimes \gamma (e_a'')
\\
&=\phi(1\otimes e_a), 
\\
\phi(e^a\otimes 1)&=\sum \langle  (e^a)', 1 \rangle (e^a)'''\otimes 1\otimes \left(\bar \gamma^* (e^a)'' \right)\otimes 1.
\\
&=\sum   (e^a)''\otimes 1 \otimes \bar \gamma^* ((e^a)')\otimes 1
\\
&=\phi(e^a\otimes 1). 
\end{align*}

Thus we have the assertion.
\end{proof}

Set 
\begin{align*}
\hat R&= \phi^{\otimes 2} (R)=\sum_a \phi(1\otimes e_{a})\otimes \phi(e^a\otimes 1)
\\
&=\sum  1\otimes e_a'\otimes 1 \otimes \gamma (e_a'')\otimes  (e^a)''\otimes 1 \otimes \bar \gamma^* ((e^a)')\otimes 1
& \in \left(H(A)\otimes H(A)^{\op}\right)^{\otimes 2}.
\end{align*}
Since $\phi^{\otimes 2}$ is an algebra homomorphism, the element $\hat R$ also satisfies the quantum Yang-Baxter equation:
\begin{align}\label{qy}
\hat R_{12}\hat R_{13}\hat R_{23}=\hat R_{23}\hat R_{13}\hat R_{12},
\end{align}
where we use the notation (\ref{subs}) treating  $H(A)\otimes H(A)^{\op}$ as one algebra. If we treat $H(A)\otimes H(A)^{\op}$ as the tensor of $H(A)$ and $H(A)^{\op}$, we have 
 \begin{align*}
\hat R_{1234}\hat R_{1256}\hat R_{3456}=\hat R_{3456}\hat R_{1256}\hat R_{1234}.
\end{align*}

Set
\begin{align*}
\tilde  e_a :=\gamma (e_a), \quad \tilde  {e}^b:=\bar \gamma^*(e^b),
\end{align*}
and set
\begin{align*}
S'&=\sum ( 1\otimes \tilde  e_{a})\otimes (e^{a}\otimes 1)\quad \in H(A)^{\op}\otimes H(A), 
\\
 S''&=\sum (1\otimes e_{a})\otimes ( \tilde  e^{a} \otimes 1)\quad \in H(A)\otimes H(A)^{\op},
\\
\tilde S&=\sum (1\otimes \tilde  e_{a})\otimes ( \tilde  e^{a} \otimes 1)\quad \in H(A)^{\op}\otimes H(A)^{\op}.
\end{align*}

Kashaev showed the following.

\begin{theorem}[Kashaev \cite{Kas97}]\label{RvsSth}
We have 
\begin{align}\label{ybe2}
\hat  R&=S_{14}''S_{13}\tilde S_{24} S_{23}' \quad \in \left(H(A)\otimes H(A)^{\op}\right)^{\otimes 2}.
\end{align}
\end{theorem}

\begin{proposition}[{Kashaev \cite{Kas97}}]\label{ys}
The quantum Yang-Baxter equation (\ref{qy}) in $\left(H(A)\otimes H(A)^{\op}\right)^{\otimes 3}$ is a consequence of the following variations of the pentagon equation for the tensors $S, S', S''$ and $\tilde S$:
\begin{align}
&S_{23}S_{12}=S_{12}S_{13}S_{23},  \label{pe1}
\quad S_{23}S'_{12}=S'_{12}S'_{13}S_{23}, \\
&S''_{23}S_{12}=S_{12}S''_{13}S''_{23}, \quad S''_{23}S'_{12}=S'_{12}S''_{13}S''_{23}, 
\end{align}
and 
\begin{align}
&S'_{23}S_{13}S''_{12}=S''_{12}S'_{23}, \quad S'_{23}S'_{13}\tilde S_{12}=\tilde S_{12}S'_{23},
\\
&\tilde S_{23}S''_{13}S''_{12}=S''_{12}S''_{23}, \quad \tilde S_{23}\tilde S_{13}\tilde S_{12}=\tilde S_{12}\tilde S_{23}.\label{pe8}
\end{align}
\end{proposition}

\section{Reconstruction of the universal quantum invariant}\label{j''}

Let $D(A)$ be the Drinfeld double of $A$. Recall from (\ref{u}) the element $u=\sum \gamma(\beta)\alpha=\sum \bar \gamma^*(e^a)\otimes e_a$ with $R=\sum \alpha \otimes \beta = \sum (1\otimes e_a)\otimes (e^a \otimes 1)$. We have a ribbon Hopf algebra 
$$D(A)^{\theta}=D(A)[\theta]/\left( \theta^2-u\gamma(u)\right)$$
with the ribbon element $\theta$ (e.g., \cite{Ka95}).

We also consider the algebra
$$\left(H(A)\otimes H(A)^{\op}\right)^{\bar \theta}=\left(H(A)\otimes H(A)^{\op}\right)[\bar \theta]/\left( \bar \theta^2-\phi (u\gamma(u))\right),$$
and extend $\phi$ to the map 
$$
\phi \co D(A)^{\theta} \to \left(H(A)\otimes H(A)^{\op}\right)^{\bar \theta}
$$
by $\phi (\theta)=\bar \theta$.

In this section, we define  tangle invariant $J''$ using $\left(H(A)\otimes H(A)^{\op}\right)^{\bar \theta}$, which turns out to be the image of tensor power of $\phi$ of the universal invariant associated to $D(A)^{\theta}$ (Theorem \ref{jj'}).

In what follows, for simplicity, we use the notation 
\begin{align*}
fx=f\otimes x\in A^*\otimes A,
\end{align*}
for $f\in A^*$ and $x\in A$. In particular we have
\begin{align*}
S=\sum_a e^a\otimes e_a, \quad S'=\sum_a \tilde e^a\otimes e_a,
\quad S''= \sum_a  e^a\otimes \tilde e_a, \quad \tilde S=\sum_a \tilde e^a\otimes \tilde e_a.
\end{align*}

\subsection{Reconstruction of the universal quantum invariant using the Heisenberg double}\label{reco}

Let $T=T_1\cup \cdots \cup T_n$ be an $n$-component, framed, ordered tangle.
Similarly to Section \ref{univinv}, set 
\begin{align*}
N_{\left(H\otimes H^{\op}\right)^{\bar \theta}}&=\mathrm{Span}_k\{ab-ba \ |\ a,b \in \left(H(A)\otimes H(A)^{\op}\right)^{\bar \theta}\}\subset \left(H(A)\otimes H(A)^{\op}\right)^{\bar \theta}.
\end{align*}
For $i=1,\ldots, n$, let 
\begin{align}\label{cont}
\left(H(A)\otimes H(A)^{\op}\right)^{\bar \theta}_i=
\begin{cases}
(H(A)\otimes H(A)^{\op})^{\bar \theta} \quad &\text{if} \quad \partial T_i \not=  \emptyset,
\\
\left(H(A)\otimes H(A)^{\op}\right)^{\bar \theta}/N_{\left(H(A)\otimes H(A)^{\op}\right)^{\bar \theta}} \quad &\text{if} \quad \partial T_i=  \emptyset.
\end{cases}
\end{align}

Take a diagram $D$ of $T$.
We define an element $J'(D)\in  \bigotimes_i\left(H(A)\otimes H(A)^{\op}\right)^{\bar \theta}_i$ modifying the definition of $J(T)$ as follows.

We duplicate $D$ and thicken the left strands following the orientation, and  denote the result by $\zeta(D)$.
See (a), (b) in Figure \ref{c} and Figure \ref{fig8} for examples.

Then we put labels on crossings as in Figure \ref{fig:color2}, where each $\gamma'$ and each $(\bar \gamma^*)'$  should be replaced with $\gamma$ and $\bar \gamma^*$, respectively, 
if the string is oriented upwards, and with the identities otherwise.

\begin{figure}
\centering
\includegraphics[width=8cm,clip]{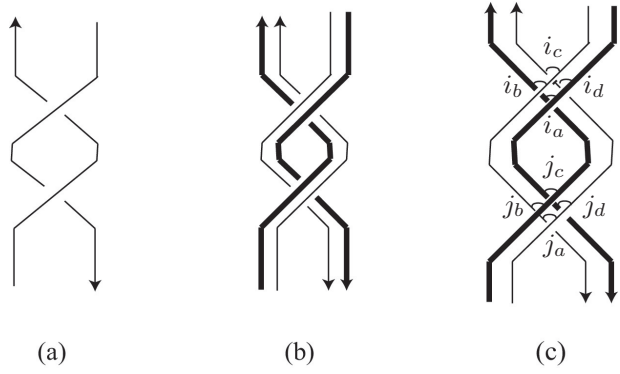}
\caption{(a) A tangle $C$, (b) The diagram $\zeta(C)$, (c) Parameters for $\zeta(C)$}\label{c}

\begin{picture}(0,0)
\put(86,133){$i_a$}
\put(71,148){$i_b$}
\put(86,163){$i_c$}
\put(101,148){$i_d$}

\put(86,88){$j_a$}
\put(71,103){$j_b$}
\put(86,118){$j_c$}
\put(101,103){$j_d$}
\end{picture}
\end{figure}

\begin{figure}
\centering
\includegraphics[width=10cm,clip]{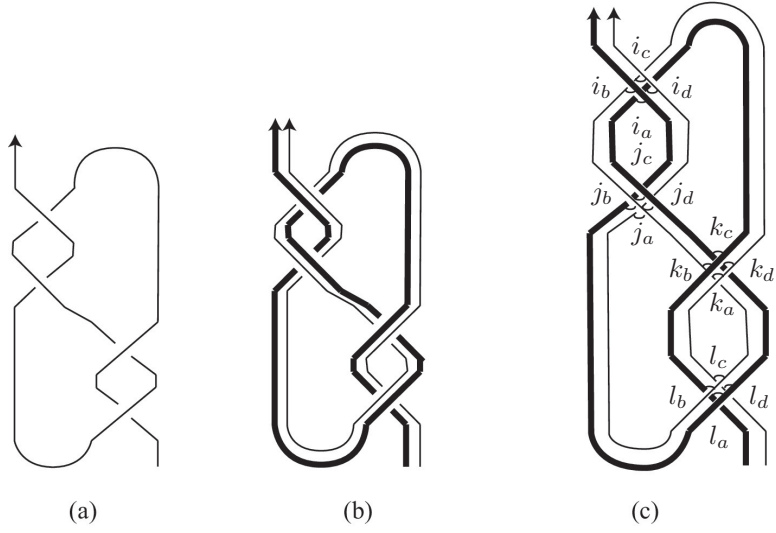}
\caption{(a) A tangle $T_{41}$, (b) The  diagram $\zeta(T_{41})$, (c) Parameters for $\zeta(T_{41})$ }\label{fig8}

\begin{picture}(0,0)
\put(91,192){$i_a$}
\put(76,207){$i_b$}
\put(91,222){$i_c$}
\put(106,207){$i_d$}

\put(91,152){$j_a$}
\put(76,167){$j_b$}
\put(91,182){$j_c$}
\put(106,167){$j_d$}

\put(120,125){$k_a$}
\put(105,138){$k_b$}
\put(120,154){$k_c$}
\put(135,138){$k_d$}

\put(120,75){$l_a$}
\put(105,90){$l_b$}
\put(120,105){$l_c$}
\put(135,90){$l_d$}
\end{picture}
\end{figure}

\begin{figure}
\includegraphics[width=11cm,clip]{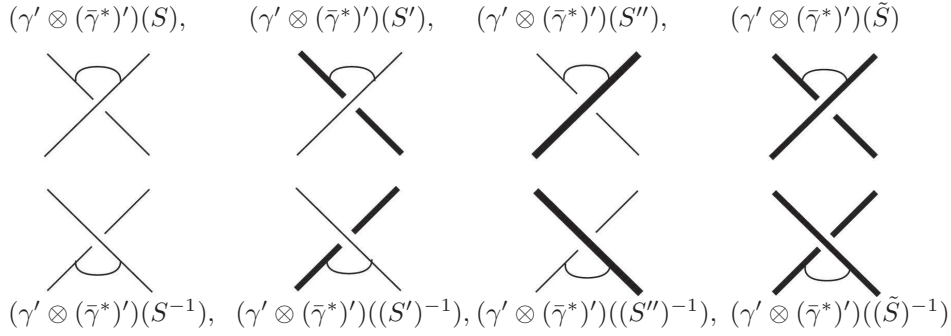}
\caption{Labels on crossings}\label{fig:color2}

\begin{picture}(0,0)
\put(-170,138){$(\gamma' \otimes (\bar \gamma^*)' )(S),$}
\put(-80,138){$(\gamma' \otimes (\bar \gamma^*)' )(S'),$}
\put(5,138){$(\gamma' \otimes (\bar \gamma^*)' )(S''),$}
\put(100,138){$(\gamma' \otimes (\bar \gamma^*)' )(\tilde S)$}

\put(-170,27){$(\gamma' \otimes (\bar \gamma^*)' )(S^{-1}),$}
\put(-85,27){$(\gamma' \otimes (\bar \gamma^*)' )((S')^{-1}),$}
\put(5,27){$(\gamma' \otimes (\bar \gamma^*)' )((S'')^{-1}),$}
\put(100,27){$(\gamma' \otimes (\bar \gamma^*)' )((\tilde S)^{-1})$}
\end{picture}
\end{figure}

We define the $(2i-1)$st and the $2i$th tensorands  of $J'(D)$  as the product 
of the labels on the thin and the thick strands, respectively, obtained by duplicating  $T_i$, where the labels are read off 
reversing the orientation, and  written from left to right. Here, if $T_i$ is a closed component, then we choose a point $p$ on $T_i$ and 
denote by $p'$ (resp. $p''$) the  image of $p$ by the duplicating procedure on the thin (resp. thick)  strand.
We read the labels of the thin (resp. thick) strand from  $p'$ (resp. $p''$).

Let 
\begin{align}\label{diagleft}
(\leftarrow)\co \text{\{tangle diagrams\}} \to \text{\{tangle diagrams\}}, \quad D \mapsto D_{(\leftarrow)},
\end{align}
where $D_{(\leftarrow)}$ is the diagram
obtained from $D$ by replacing each of \includegraphics[width=0.5cm,clip]{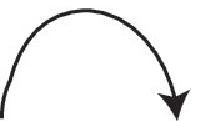} and \includegraphics[width=0.5cm,clip]{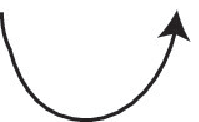} with \includegraphics[width=0.5cm,clip]{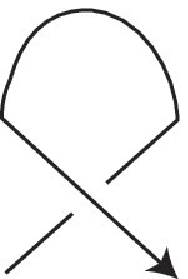} and  \includegraphics[width=0.5cm,clip]{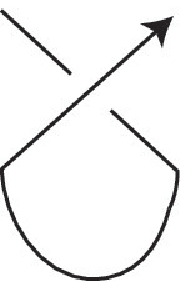}, respectively.

For $i=1,\ldots, n$, let $D_i$ be the subdiagram of $D$ corresponding to $T_i$. We define $d(D_i)$ as the number of \includegraphics[width=0.5cm,clip]{max.eps}  minus the number of \includegraphics[width=0.5cm,clip]{min.eps}  in $D_i$.

Set
\begin{align*}
J''(D)=\left(\prod_i \bar \theta_i^{d(D_i)}\right)J'(D_{(\leftarrow)}) \in \bigotimes_i \left(H(A)\otimes H(A)^{\op}\right)^{\bar \theta}_i,
\end{align*}
where $\bar \theta_i$ is defined following the notation (\ref{subs}).

 \begin{theorem}\label{jj'}
We have 
 \begin{align*}
 \phi^{\otimes n}\circ J(T)&=J''(D) \quad \in \bigotimes_i \left(H(A)\otimes H(A)^{\op}\right)^{\bar \theta}_i.
 \end{align*}
   
  \end{theorem}
  
If moreover  $T$ is a braid, which is a $0$-framed tangle with no maxima or minima, then we have $D=D_{(\leftarrow)}$, $\prod_i \bar \theta_i^{d(D_i)}=1$. Thus we have the following.

  \begin{corollary}
If  $T$ is a braid, then we have 
  \begin{align*}
\phi^{\otimes n}\circ J(T)= J'(D) \quad \in \bigotimes_i \left(H(A)\otimes H(A)^{\op}\right)_i,
  \end{align*}
  where $ \left(H(A)\otimes H(A)^{\op}\right)_i$ is defined similarly as (\ref{cont}) using $H(A)\otimes H(A)^{\op}$
  instead of $\left(H(A)\otimes H(A)^{\op}\right)^{\bar \theta}$.
  \end{corollary}
 
  Let $
  f(D_i)=\#\text{\{positive self crossings of $D_i$\}}-\#\text{\{negative self crossings of $D_i$\}}$
 be the framing  of $D_i$.
Set $$
J^0(D)=\left(\prod_i \bar \theta_i^{f(D_i)}\right) J'(D) \in \bigotimes_i \left(H(A)\otimes H(A)^{\op}\right)^{\bar \theta}_i.
$$

We use the following lemma to prove Theorem \ref{jj'}.

\begin{lemma}\label{jj'l}
 Let $T$ be an $n$-component framed tangle, and let $T^0$ denote $T$ with $0$-framing.
 Let $D$ be a diagram of $T$.
 We have 
\begin{align*}
\phi^{\otimes n}\circ J(T^0) = J^0(D_{(\leftarrow)})\quad \in \bigotimes_i \left(H(A)\otimes H(A)^{\op}\right)^{\bar \theta}_i.
 \end{align*}
 \end{lemma}
 \begin{proof}
 For a positive (resp. negative) crossing $c=c_1\cup c_2$, where $c_1$ is the under strand, let $c^0$ be a tangle obtained by inserting a negative (resp. positive) kink to the bottom of $c_1$, see    Figure \ref{fig:normalcross} for examples. 
We take a diagram $D^0$ of $T^0$ obtained from $D_{(\leftarrow)}$ by replacing each self crossing $c$ by $c^0$ so that framings vanish.  

\begin{figure}
\centering
\begin{picture}(0,50)
\put(-110,0){\includegraphics[width=8cm,clip]{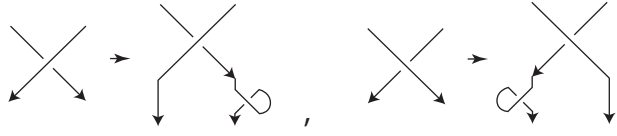}}
\end{picture}
\caption{How to insert a kink to a crossing}\label{fig:normalcross}

\end{figure}

Since the labels on  $D_{(\leftarrow)}$ to define $J'$ are only on crossings (since there are no \includegraphics[width=0.5cm,clip]{max.eps} and \includegraphics[width=0.5cm,clip]{min.eps}), 
in order to prove the assertion it is enough to show
\begin{enumerate}
\item $\phi^{\otimes 2}\circ J(c_{\pm})=J'(c_{\pm})$,
\item $\phi^{\otimes 2}\circ J(c^0_{\pm})=\bar \theta^{\pm 1}_1 J'(c_{\pm})$,
\end{enumerate}
for positive (resp. negative) crossing $c_+$ (resp. $c_-$) with each strand oriented arbitrarily.

Assume that each strand of $c_{\pm}$ is oriented downwards.
 
(1) follows from
\begin{align*}
\phi^{\otimes 2}\circ J(c_{+})&=\hat R_{1234}=S_{14}''S_{13}\tilde S_{24} S_{23}'
\\&=\sum_{a,b,c,d}e_{a}e_{b} \otimes \tilde  e_{d} \tilde  e_{c} \otimes e^{b}  e^{c}\otimes  \tilde  e^{a} \tilde  e^{d}=J'(c_{+}),
\\
\phi^{\otimes 2}\circ J(c_{-})&=\hat R^{-1}_{1234}=(S_{23}')^{-1}(\tilde S_{24} )^{-1}(S_{13})^{-1}(S_{14}'')^{-1}
\\
&=\sum_{a,b,c,d} u_{b} u_{c}\otimes\tilde  u_{a}\tilde  u_{d}
 \otimes  u^{a} u^{b}\otimes
 \tilde  u^{d}\tilde  u^{c}=J'(c_{-}),
\end{align*}
where $u_a$, $u^a$, $\tilde u_a$, $\tilde u^a$ are defined by 
 \begin{align*}
\sum_a u_a\otimes u^a=S^{-1}=\sum_a\gamma (e_a)\otimes e^a,\quad \sum_a \tilde  u_a\otimes u^a= (S')^{-1}=\sum_a\gamma (\tilde  e_a)\otimes e^a,
\\
\sum_a u_a\otimes \tilde  u^a=(S'')^{-1}=\sum_a\gamma (e_a)\otimes \tilde  e^a, 
\quad \sum_a \tilde  u_a\otimes  \tilde  u^a=\tilde S ^{-1}=\sum_a\gamma (\tilde  e_a)\otimes \tilde  e^a,
\end{align*}
see Figure \ref{fig:crossSbc}.

\begin{figure}
\includegraphics[width=5cm,clip]{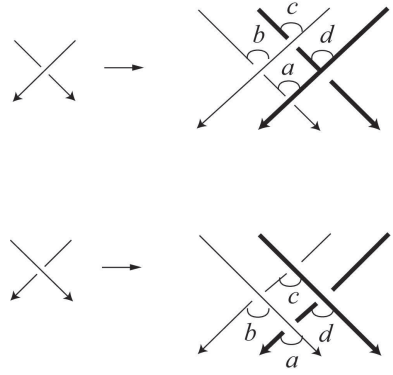}
\caption{Labels on the colored diagrams $\zeta(c_{\pm})$ associated to  positive and negative crossings $c_{\pm}$}\label{fig:crossSbc}
\end{figure}

Since the universal invariant of a positive (resp. negative) kink is equal to $\theta^{-1}$ (resp. $\theta$), we have  $J(c_{\pm}^0)=\theta^{\pm 1}_1 J(c_{\pm})$. Thus (2) follows from 
\begin{align*}
\phi^{\otimes 2}\circ J(c_{\pm}^0)&=\phi^{\otimes 2}\left(\theta^{\pm 1}_1 J(c_{\pm})\right)
\\
&=\bar \theta^{\pm 1}_1(\phi^{\otimes 2}\circ J)(c_{\pm})
\\
&=\bar \theta^{\pm 1}_1 J'(c_{\pm}),
\end{align*}
where the last identity follows from (1). 

For a crossing $c_{\pm}$ with other orientations,  (1) and (2) follow similarly from
\begin{align*}
\phi^{\otimes 2}\circ (\gamma_{D(A)} \otimes 1)(R)&=\sum_{a,b,c,d} \gamma(\tilde  e_{c})\gamma (\tilde  e_{d}) \otimes  \gamma(e_{b})  \gamma(e_{a})\otimes e^{b}  e^{c}\otimes  \tilde  e^{a}\tilde  e^{d},
\\
\phi^{\otimes 2}\circ (1\otimes \gamma_{D(A)} )(R)&=\sum_{a,b,c,d}e_{a}e_{b} \otimes \tilde  e_{d} \tilde  e_{c} \otimes \bar \gamma^* ( \tilde  e^{d}) \bar \gamma^* (\tilde  e^{a})\otimes  \bar \gamma^* (e^{c}) \bar \gamma^* (e^{b}) ,
\\
\phi^{\otimes 2}\circ (\gamma_{D(A)} \otimes 1)(R^{-1})&=\sum_{a,b,c,d} \gamma(\tilde  u_{d}) \gamma( \tilde  u_{a})\otimes  \gamma(u_{c})\gamma(u_{b}) \otimes  u^{a}u^{b}\otimes \tilde  u^{d}\tilde  u^{c},
\\
\phi^{\otimes 2}\circ (1\otimes \gamma_{D(A)})(R^{-1})&= \sum_{a,b,c,d}u_{b}u_{d}\otimes \tilde  u_{a}\tilde  u_{d} \otimes   \bar \gamma^* ( \tilde  u^{c}) \bar\gamma^* (\tilde  u^{d}) \otimes \bar \gamma^* (u^{b}) \bar \gamma^* (u^{a}),
\end{align*}
which completes the proof.

\end{proof}

\begin{proof}[Proof of Theorem \ref{jj'}]

By Lemma \ref{jj'l} we have
  \begin{align*}
 \phi^{\otimes n}\circ J(T)&= \phi^{\otimes n}\left( \left(\prod_i  \theta_i^{-f(D_i)}\right)  J(T^0)\right)
 \\&= \left(\prod_i  \bar \theta_i^{-f(D_i)}\right) \left( \phi^{\otimes n} \circ J(T^0)\right)
\\&= \left(\prod_i  \bar \theta_i^{-f(D_i)}\right) J^0(D_{(\leftarrow)})
\\&=\left(\prod_i \bar  \theta_i^{-f(D_i)}\right) \left(\prod_i  \bar\theta_i^{f((D_{(\leftarrow)})_i)}\right) J'(D_{(\leftarrow)})
\\&=\left(\prod_i \bar \theta_i^{d(D_i)}\right)J'(D_{(\leftarrow)}).
 \end{align*}

Thus we have the assertion.

\end{proof}

For the example with $C$, with the parameters as in Figure \ref{c} (c), we have
\begin{align*}
J'_{C}=\sum_{i_a,i_b,i_c,i_d,j_a,j_b,j_c,j_d} &\gamma(  e_{i_c})\gamma( e_{i_d})\bar \gamma^*( e^{j_d}) \bar \gamma^*(  e^{j_a})\otimes\gamma(\tilde e_{i_b})\gamma(\tilde e_{i_a}) \bar \gamma^*(\tilde e^{j_c})\bar \gamma^*(\tilde e^{j_b})
\\
&\otimes e_{j_a}e_{j_b}e^{i_b}e^{i_c}\otimes \tilde  e_{j_d}\tilde  e_{j_c}\tilde  e^{i_a}\tilde  e^{i_d}.
\end{align*}
For the example with $T_{41}$, with the parameters as in Figure \ref{fig8} (c), we have
\begin{align*}
J'_{{T_{41}}}=&\sum_{i_a,i_b,i_c,i_d,j_a,j_b,j_c,j_d, k_a, k_b,k_c,k_d, l_a,l_b,l_c,l_d}\bar \gamma^*(u^{i_c})\bar \gamma^*(u^{i_d})\gamma(u_{j_d})\gamma(u_{j_a})e^{l_b}e^{l_c}
e_{k_a}e_{k_b}
\\
&\times u^{j_a}u^{j_b}u_{i_b}u_{i_c}\bar \gamma^*(e^{k_d})\bar \gamma^*(e^{k_a})\gamma( e_{l_c})\gamma(e_{l_d})
\\
&\otimes 
\bar \gamma^*(\tilde  u^{i_b})\bar \gamma^*(\tilde  u^{i_a})\gamma(\tilde  u_{j_c})\gamma(\tilde  u_{j_b})
 \tilde e^{l_a} \tilde e^{l_d}\tilde e_{k_d} \tilde e_{k_c}
\\
&\times\tilde  u^{j_d}\tilde   u^{j_c}\tilde   u_{i_a}\tilde   u_{i_d}\bar \gamma^*(\tilde  e^{k_c})\bar \gamma^*(\tilde  e^{k_b})\gamma (\tilde  e_{l_b})\gamma(\tilde  e_{l_a}).
\end{align*}

\section{Extension of the universal quantum invariant to an invariant for colored diagrams}\label{exte}

In this section we define \textit{colored diagrams} and extend the map $J'$ to an invariant for colored diagrams up to  \textit{colored moves}.

\subsection{Colored diagrams and an extension of $J'$.}\label{coloredd}

In what follows, we consider also a virtual crossing as in Figure \ref{generator0}, which we call a \textit{symmetry}.
By a \textit{crossing} we mean only a real crossing.

A \textit{colored diagram} $Z$ is a virtual tangle diagram consisting of \textit{thin} strands  and \textit{thick} strands, which is obtained by pasting, horizontally and vertically, copies of fundamental tangle diagrams in Figure \ref{fig:fundamental} and copies of the symmetry, where the thickness of each strand are arbitrary.

\begin{figure}
\includegraphics[width=1.5cm,clip]{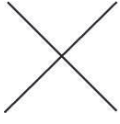}
\caption{A symmetry, where the orientation  of each strand is arbitrary}\label{generator0}
\end{figure}

Let $\mathcal{CD}$ be the set of colored diagrams.
For $\mu=(\mu_1\ldots, \mu_n), \nu=(\nu_1, \ldots, \nu_n) \in \{\pm\}^n$, we denote by
\begin{align*}
\mathcal{CD}(\mu; \nu) \subset \mathcal{CD}
\end{align*}
the set of $n$-component colored diagrams $Z=Z_1\cup \cdots \cup Z_n$ such that 
\begin{align*}
Z_i \text{ is thin } \Leftrightarrow \mu_i=+,& \quad Z_i \text{ is thick } \Leftrightarrow  \mu_i=-,
\\
\partial Z_i \not= \emptyset \Leftrightarrow \nu_i=+,& \quad \partial Z_i = \emptyset \Leftrightarrow \nu_i=-.
\end{align*}

For $i=1,\ldots,n$, set
\begin{align*}
H(A)_i^{+}&=H(A),
\quad 
H(A)_i^{-}=H(A)/N_{H(A)},
\\
 (H(A)^{\op})_i^{+}&=
H(A)^{\op}, \quad
 (H(A)^{\op})_i^{-}=
H(A)^{\op}/N_{H(A)^{\op}}.
\end{align*}

We define the map 
\begin{align*}
J'\co \mathcal{CD}(\mu; \nu)  \to  \bigotimes_{\mu_i=+} H(A)_i^{\nu_i} \bigotimes_{\mu_j=-} (H(A)^{\op})_j^{\nu_j}
\end{align*}
 in a similar way to the definition of $J'$ in Section \ref{j''}, i.e., by putting the labels on the crossings as in Figure \ref{fig:color2}, not putting label for other fundamental tangle diagrams, and by taking the product of the labels.

\subsection{Colored moves}\label{cm}

We define several moves on colored diagrams as follows.

The \textit{colored Pachner $(2,3)$ moves} are defined in Figure  \ref{fig:sandp2++}.
Note that each colored Pachner $(2,3)$ move involves a symmetry, and thus is not the Reidemeister III move on tangle diagrams.
 
The \textit{colored  $(0,2)$ moves} are defined in  Figure \ref{fig:02}.

The \textit{symmetry moves} are defined in Figure \ref{fig:symmetrycolor}.

The \textit{planar isotopies}  are defined in Figure \ref{fig:isotopy}.
\footnote{It is known that if two tangle diagrams $D$ and $D'$ are planer isotopic to each other, then $D$ and $D'$  are related by a sequence of the moves defined  in Figure \ref{fig:isotopy}, see e.g., \cite{Kas95}.}


\begin{figure}
\includegraphics[width=9cm,clip]{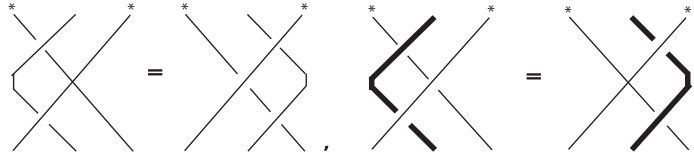}

\caption{The colored Pachner $(2,3)$ moves, where  the orientation  and the thickness of each $*$-marked strand are arbitrary}\label{fig:sandp2++}
\end{figure}

\begin{figure}
\includegraphics[width=3cm,clip]{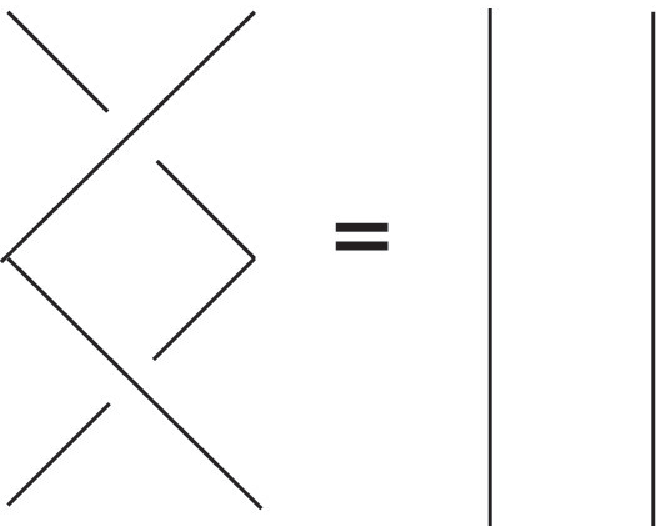}
\caption{The colored  $(0,2)$ moves,  where  the orientation and  the thickness of each strand are  arbitrary}\label{fig:02}
\end{figure}

\begin{figure}
\includegraphics[width=8cm,clip]{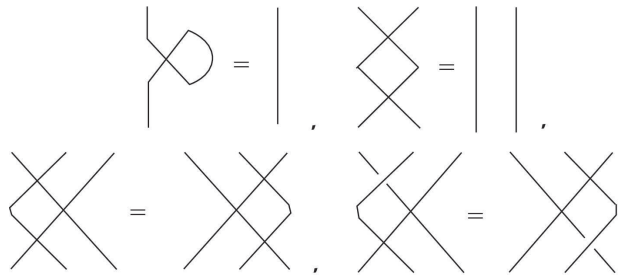}

\caption{The symmetry moves, where  the orientation and thickness of each strand are  arbitrary}\label{fig:symmetrycolor}
\end{figure}

 \begin{figure}
\includegraphics[width=13cm,clip]{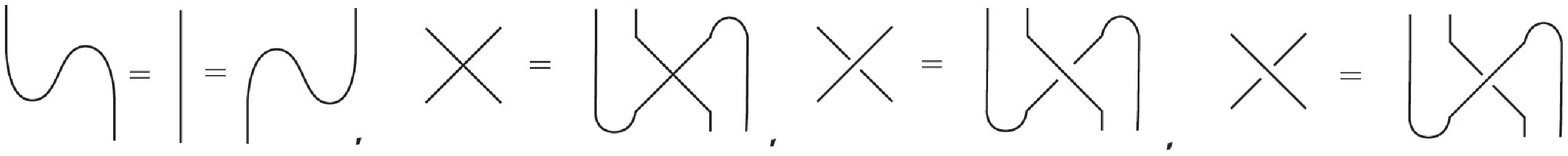}
\caption{The planar isotopies, where the orientation and thickness of each strand are arbitrary}\label{fig:isotopy}
\end{figure}

We call each of the above move a \textit{colored move}.

Let $\sim_c$  be the equivalence relation on the set of colored diagrams  generated by all colored moves.

Similarly, let  $\sim'_c$ be the equivalence relation on the set of colored diagrams  generated by colored moves except for the moves in Figure \ref{fig:exception}.

\begin{figure}
\includegraphics[width=9cm,clip]{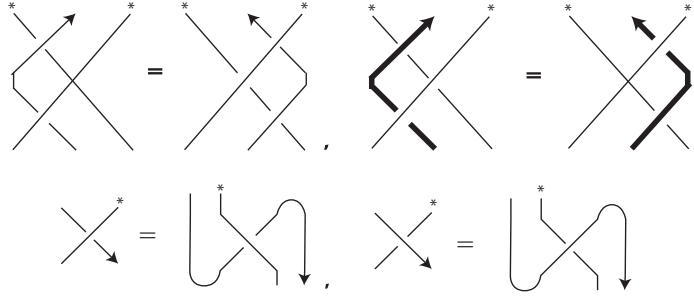}

\caption{The colored moves which are not in generators for $\sim'_c$, where the orientation  and the thickness of each $*$-marked strand are arbitrary}\label{fig:exception}
\end{figure}

We have the following.

\begin{theorem}\label{sakie}
The map $J'$ is an invariant under  $\sim'_c$.
If $\gamma^2=1$, then the map $J'$ is also an invariant under $\sim_c$.
\end{theorem}
\begin{proof}
Let  $Z$ and $Z'$ be two colored diagram.

If  $Z$ and $Z'$ are related by a colored Pachner $(2,3)$ move with strands oriented downwards,
then $J'(Z)=J'(Z')$ follows from the pentagon relations (\ref{pe1})--(\ref{pe8}).
If some $*$-marked strands are upwards, then $J'(Z)=J'(Z')$ follows from the pentagon relations, after applying the antipode on each tensorand corresponding to an upward strand.

If  $Z$ and $Z'$ are related by a colored $(0,2)$ move,
then $J'(Z)=J'(Z')$ follows from the invertibility of $S$, $S'$, $S''$, and $\tilde S$.

If  $Z$ and $Z'$ are related by a symmetry move, or by a planar isotopy which does not involve a crossing,
then it is easy to see $J'(Z)=J'(Z')$.

Let us assume that $Z$ and $Z'$ are related by a planar isotopy which involves a crossing.
If the planar isotopy is not in Figure \ref{fig:exception}, then $J'(Z)=J'(Z')$ follows from
\begin{align*}
&(\gamma \otimes 1)(S)=S^{-1}, \quad (1\otimes \bar \gamma^*)(S^{-1})=S, \quad (\gamma \otimes \bar \gamma ^*)(S)=S \quad  \in  H(A)\otimes H(A), 
\\
&(\gamma \otimes 1)(S')=(S')^{-1}, \quad (1\otimes \bar \gamma^*)((S')^{-1})=S', \quad (\gamma \otimes \bar \gamma ^*)(S')=S' \quad  \in  H(A)^{\op}\otimes H(A),
\\
&(\gamma \otimes 1)(S'')=(S'')^{-1}, \quad (1\otimes \bar \gamma^*)((S'')^{-1})=S'', \quad (\gamma \otimes \bar \gamma ^*)(S'')=S'' \quad  \in H(A)\otimes H(A)^{\op},
\\
&(\gamma \otimes 1)(\tilde S)=\tilde S^{-1}, \quad (1\otimes \bar \gamma^*)(\tilde S^{-1})=\tilde S, \quad (\gamma \otimes \bar \gamma ^*)(\tilde S)=\tilde S \quad  \in H(A)^{\op}\otimes H(A)^{\op}.
\end{align*}
If the planar isotopy is in Figure \ref{fig:exception}, then we have $J'(Z)=J'(Z')$ if $\gamma^2=1$, by
\begin{align*}
&(1 \otimes \bar \gamma^*)(S)=S^{-1}, \quad (\gamma \otimes 1)(S^{-1})=S  \quad  \in  H(A)\otimes H(A), 
\\
&(1 \otimes \bar \gamma^*)(S')=(S')^{-1}, \quad (\gamma \otimes 1)((S')^{-1})=S'  \quad  \in  H(A)^{\op}\otimes H(A),
\\
&(1 \otimes \bar \gamma^*)(S'')=(S'')^{-1}, \quad (\gamma \otimes 1)((S'')^{-1})=S''  \quad  \in H(A)\otimes H(A)^{\op},
\\
&(1 \otimes \bar \gamma^*)(\tilde S)=\tilde S^{-1}, \quad (\gamma \otimes 1)(\tilde S^{-1})=\tilde S  \quad  \in H(A)^{\op}\otimes H(A)^{\op}.
\end{align*}

If  $Z$ and $Z'$ are related by a  colored Pachner $(2,3)$ move in Figure \ref{fig:exception}, i.e., a  colored Pachner $(2,3)$ move with middle strands oriented upwards, then  $Z$ and $Z'$ are related by planer isotopy and the colored Pachner $(2,3)$ move with middle strands oriented downwards.  Thus we have  $J'(Z)=J'(Z')$ by the above argument.

Thus we have the assertion.
\end{proof}

\subsection{Tangles and colored diagrams}\label{tacd}
Recall from  Section \ref{reco} the diagram $\zeta(D)$ associated to a tangle diagram $D$.
Actually $\zeta(D)$ is nothing but a colored diagram and $\zeta$ defines a map 
$$
\zeta \co \text{\{tangle diagrams\}}  \to \text{\{colored diagrams\}}.
$$
Let $\sim_{RII, RIII}$ be the regular isotopy, i.e., the equivalence relation of tangle diagrams generated by Reidemeister II, III moves and planar isotopies of tangle diagrams.
We have the following.

\begin{theorem}\label{invr}
 Let $D$ and $D'$ be two diagrams such that $D\sim_{RII, RIII} D'$.
 Then we have $\zeta(D)\sim_c \zeta(D)$. 
\end{theorem}

\begin{proof}

Let $D$ and $D'$ be two tangle diagrams related by a Reidemeister II move. We can transform $\zeta (D)$ to $\zeta(D')$  by applying  colored $(0,2)$ moves four times, see Figure \ref{fig:ReiII} for the case that each strand is oriented downwards.

Let $D$ and $D'$ be  two tangle diagrams related  by a Reidemeister III move. 
We can transform  $\zeta (D)$ to $\zeta(D')$  by applying  colored Pachner $(2,3)$ moves eight times, see Figure \ref{fig:ReiIII} for the case that each strand is oriented downwards.

Let $D$ and $D'$ be two tangle diagrams which are related by the planar isotopy.
Then we can also transform  $\zeta (D)$ to $\zeta(D')$  by the planar isotopies,
see Figure \ref{fig:isotopyleft} for  examples.

\begin{figure}
\includegraphics[width=9cm,clip]{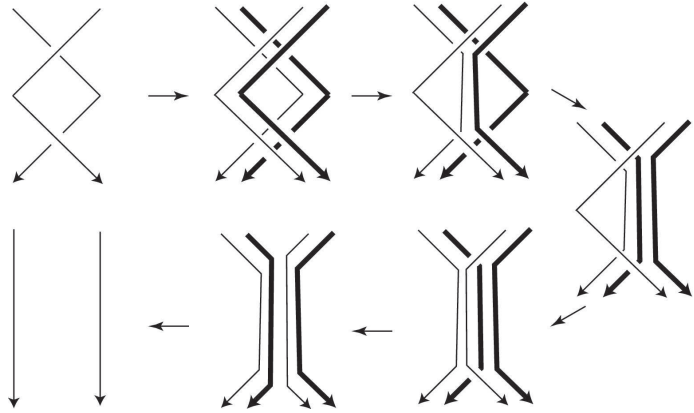}
\caption{A realization of Reidemeister II move using colored moves}\label{fig:ReiII}
\end{figure}

 \begin{figure}
\includegraphics[width=13cm,clip]{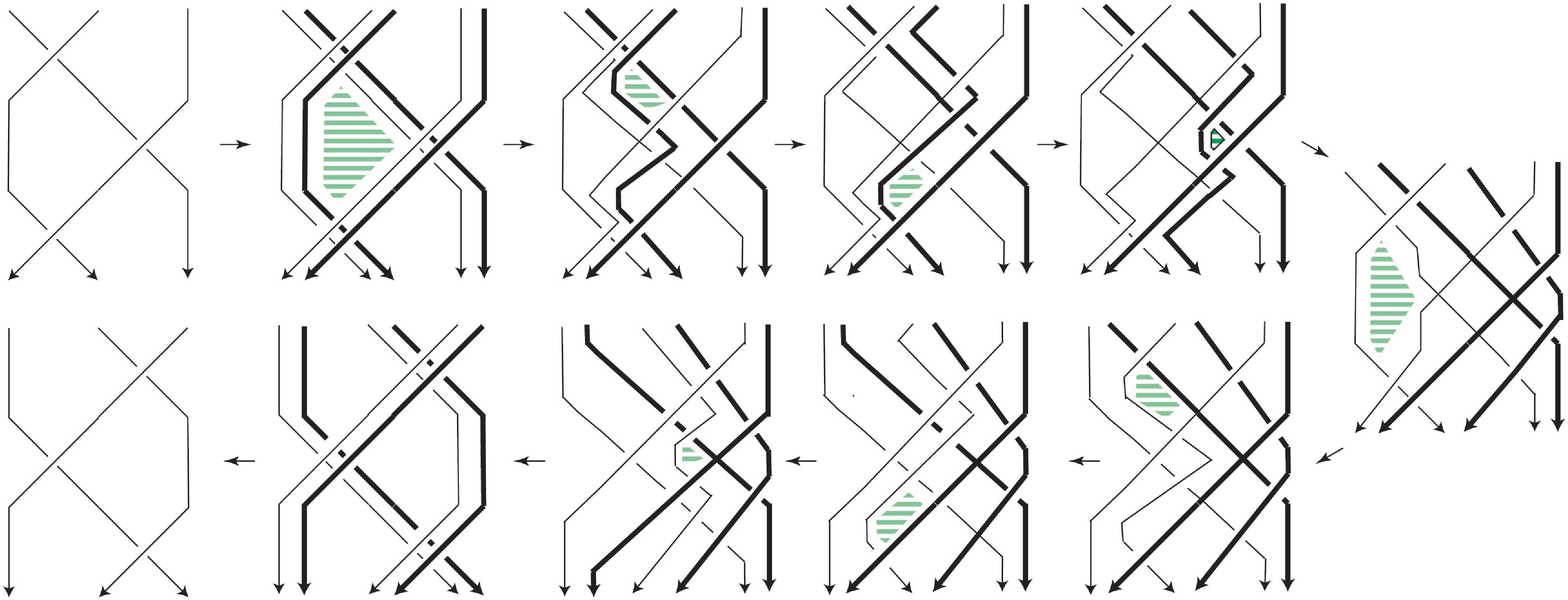}
\caption{A realization of Reidemeister III move  using colored moves}\label{fig:ReiIII}
\end{figure}


\begin{figure}
\includegraphics[width=13cm,clip]{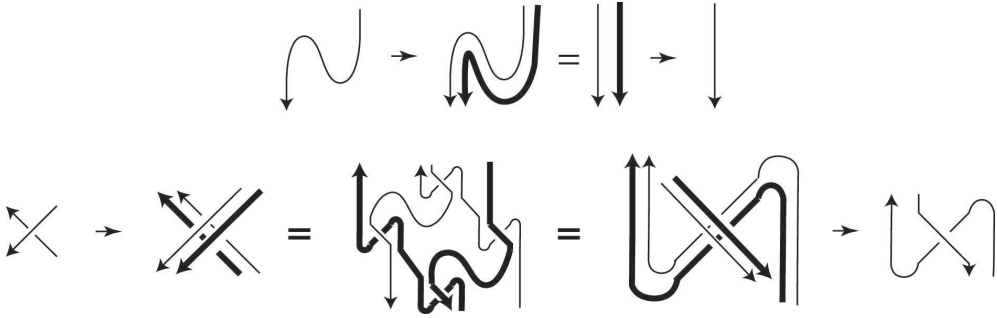}
\caption{Realizations of planar isotopies  using colored moves}\label{fig:isotopyleft}
\end{figure}

\end{proof}

Note that the diagrammatic transformations in Figure \ref{fig:ReiIII} induces algebraic equations via the universal invariant $J'$, which gives a proof of Proposition \ref{ys}. See the  Table \ref{ta} for the situation.

\begin{table}
\begin{center}
\begin{tabular}{lcl}
Reidemeister III move & $\xrightarrow[]{J'}$ & quantum Yang-Baxter equation
\\
Colored Pachner $(2,3)${ moves} & $\xrightarrow[]{J'}$  & pentagon relations
\\
Figure \ref{fig:ReiIII}
 & $\xrightarrow[]{J'}$  & Proposition \ref{ys}
 \\
  (  Colored $(2,3)${ move} $\Rightarrow$  RIII move )  & & (pentagon relation $\Rightarrow$ quantum Yang-Baxter equation)
   \end{tabular}
\end{center}
\caption{Correspondence between topological situation and algebraic situation}\label{ta}
\end{table}

\section{$3$-dimensional descriptions: colored diagrams and colored singular triangulations}\label{cit}
In this section we associate a \textit{colored tetrahedron} to each crossing of a colored diagram $Z$, and define a \textit{colored cell complex} associated to $Z$.
Using a colored cell complex we define  a \textit{colored singular triangulation} of a topological space.
As a result,  the universal quantum invariant $J'$ turns out to be an invariant of  colored singular triangulations, where
a copy of the $S$-tensor is attached to each colored tetrahedron.

\subsection{Colored tetrahedra}\label{coi}

Consider a tetrahedron $\Gamma$ in the oriented space $\mathbb{R}^3$ with an ordering of its $2$-faces $f_1,f_2,f_3,f_4$. We stick $\Gamma$ by two strands going into $\Gamma$ at $f_1$ (resp. $f_3$) and out of $\Gamma$ at $f_2$ (resp. $f_4$).  Note that there are two types of such tetrahedra up to rotation  as in Figure \ref{fig:tridiag}, where such a tetrahedra is presented by a crossing so that the strand piercing $f_1$ and $f_2$ is over. 
We consider two types of strands, depicted by thick and thin strands, and then there are eight types of such  tetrahedra, which we call  \textit{colored tetrahedra},  presented by eight types of crossings as in Figure \ref{fig:colortri}.

\begin{figure}
\includegraphics[width=10cm,clip]{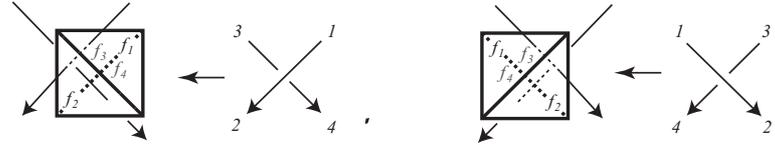}
\caption{Two types of tetrahedra which are sticked by two ordered strands}\label{fig:tridiag}
\end{figure}

\begin{figure}
\includegraphics[width=11cm,clip]{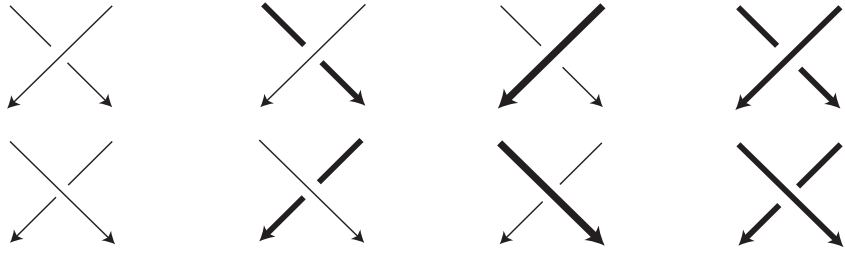}
\caption{Colored tetrahedra}\label{fig:colortri}

\end{figure}

\subsection{Colored diagrams and colored cell complexes}\label{codiag}

We define a \textit{colored cell complex} $\mathcal{C}(Z)$ associated to a colored diagram $Z$  as follows.

Recall that  $Z$ consists of fundamental tangles and symmetries.
Let  $\{c_1,\ldots c_k\}$ be the set of crossings in $Z$.
To each crossing $c_i$, associate a colored tetrahedron $\Gamma_i$ as in Section \ref{coi}.
See Figure \ref{cell} for an example.

\begin{figure}
\includegraphics[width=6cm,clip]{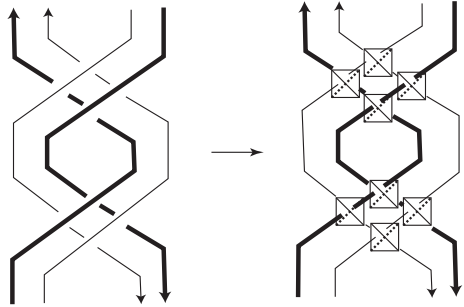}
\caption{How to associate tetrahedra  on a colored diagram}\label{cell}
\end{figure}

We define $\mathcal{C}(Z)$ to be the cell complex obtained from colored tetrahedra $\Gamma_1, \ldots ,\Gamma_k$
by gluing them along their $2$-faces as follows.
\begin{enumerate}
\item $2$-faces $F$ and $F'$ of $\Gamma_1, \ldots ,\Gamma_k$ 
are glued if and only if $F$ and $F'$ are adjacent along $Z$.
\item We mark by $*$ the vertex of each $2$-face of $\Gamma_1, \ldots ,\Gamma_k$ 
as in Figure \ref{fig:colorpaste} depending on the thickness of strands and the order of the faces in a tetrahedron, and glue adjacent faces $F$ and $F'$ so that the $*$-marked vertices are attached. 
\end{enumerate}

\begin{figure}
\includegraphics[width=8cm,clip]{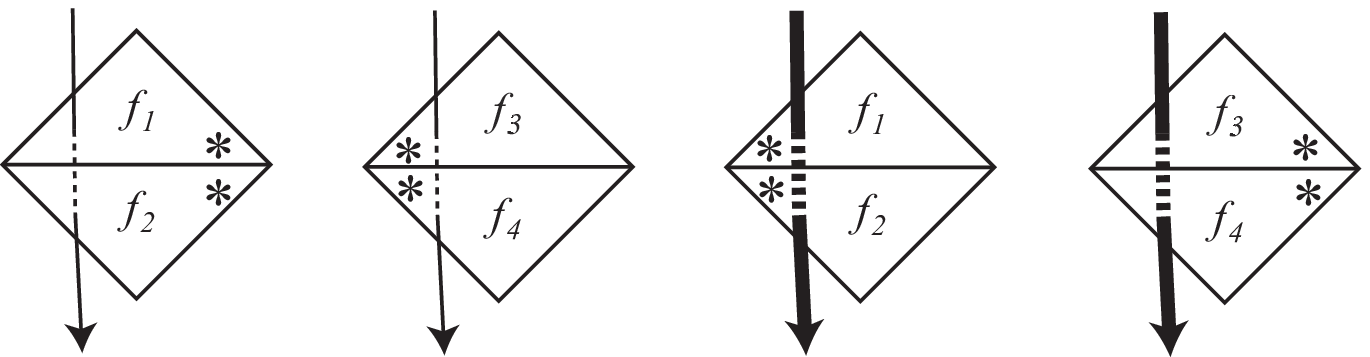}
\caption{How to mark by $*$ a vertex of each triangle}\label{fig:colorpaste}

\end{figure}

\subsection{Colored singular triangulations and colored ideal triangulations}\label{CD}

For a space $X$,  a \textit{singular triangulation} (see e.g., \cite{TV, BB1}) of $X$ consists of a finite index set $I$, a function $d\co I \to \mathbb{N}$, and continuous maps $f_{i} \co \Delta^{d(i)}\to X$ for $i\in I$, where $\Delta^{n}$ is the standard $n$ simplex, such that 
$(I, d, \{f_i\}_{i\in I})$ is a finite cell decomposition of $X$, and for each $i\in I$ and a face $F$ in $\Delta^{d(i)}$, the restriction  $f_i|_F$ is the composition $f_j \circ g$  of an affine isomorphism $g\co F \to \Delta^{d(j)}$ and $f_j$ for some $j \in I$.

Let $\mathcal{C}(Z)$ be the colored cell complex of a colored diagram $Z$, which we can naturally regard a singular triangulation.
Consider
\begin{align*}
X = \mathcal{C}(Z)/(e_1=e'_1,\ldots ,e_k=e'_k, v_1=v'_1,\ldots ,v_l=v'_l)
\end{align*}
be a singular triangulation obtained from $\mathcal{C}(Z)$ by identifying some pairs of edges $(e_1,e'_1),\ldots , (e_k,e'_k)$,  $k\ge 0$,  and some pairs of vertices $(v_1,v'_1),\ldots ,(v_l,v'_l), l\ge0,$ in $\mathcal{C}(Z)$. 
We call $X$ a \textit{colored singular triangulation} (coloring) of type $Z$.
In particular,  if $X$  is an ideal triangulation of some topological space $\tilde X$, then we call it a \textit{colored ideal triangulation} of  $\tilde X$.  

Let $\mathcal{CT}(Z)$ be the set of colored singular triangulations of type $Z$ and set
\begin{align*}
\mathcal{CT}=\bigcup_{Z\in \mathcal{CD}} \mathcal{CT}(Z).
\end{align*}


\begin{remark}\label{branch}

In this remark, we assume $3$-manifolds are connected, compact, oriented, and with non-empty boundary.

In \cite{BP, BB5, BB6},  Benedetti-Petronio and Baseilhac-Benedetti used so called \textit{$\mathcal{N}$-graphs} to represent  \textit{branched} ideal triangulations of a $3$-manifolds and dual oriented \textit{standard branched} spines of them. In this remark we consider abstract $\mathcal{N}$-graphs, i.e., we do not take planar immersions of them.

Let $\mathcal{BTR}$ the set  of branched ideal triangulations of $3$-manifolds, 
$\mathcal{BSP}$ the set  of oriented standard branched spines of $3$-manifolds, and  $\mathcal{NG}$ the set of $\mathcal{N}$-graphs with the color $0 \in \mathbb{Z}/3\mathbb{Z}$ on every edge.

We have the bijections 
\begin{align*}
\mathcal{NG} \to \mathcal{BSP}, \quad G \mapsto BSP(G),
\\
\mathcal{NG} \to \mathcal{BTR},  \quad G \mapsto BTR(G),
\end{align*}
where $BSP(G)$ is obtained from $G$ \cite{BB5} so that a $4$-valent vertex encodes a branched
tetrahedron, and  $BTR(G)$ is the  branched ideal triangulations which is the dual of the  oriented standard branched spine $BSP(G)$.

Let $\mathcal{CCD}$ be the set of equivalence classes of \textit{closed} colored diagrams up to planar isotopies and symmetry moves.
We have the surjective map
\begin{align*}
p\co \mathcal{CCD} \to \mathcal{NG}, \quad Z\mapsto p(Z)
\end{align*}
where $p(Z)$ is the $\mathcal{N}$-graph obtained from $D$ by reversing the orientation of thick strands. 

It is not difficult to check that the branched ideal triangulation $BTR(p(Z))$ is the colored singular triangulation  of type $Z$ obtained from $\mathcal{C}(Z)$ by identifying some edges and vertices so that $BSP(p(Z))$ becomes a standard spine, i.e., the complement of the vertices in the singular set of $BSP(p(Z))$ is a union of segments, and the complement of the singular set in $BSP(p(Z))$ is a union of disks.
For an example with link complements, see the proof of Proposition \ref{ci}.
\end{remark}

\subsection{Colored moves and colored singular triangulations}
We can  translate  colored moves on the set $\mathcal{CD}$ of colored diagrams defined in Section \ref{cm} to moves on  the set $\mathcal{CT}$ of colored singular triangulations  as follows.

For colored diagrams $Z$ and $Z'$, let $X$ and $X'$ be colored singular triangulations of types $Z$ and $Z'$, respectively.
Let 
\begin{align*}
\psi\co \mathcal{C}(Z) \to X, \quad \psi'\co \mathcal{C}(Z') \to X', 
\end{align*}
be the projections.
We say that $X$ and $X'$ are related by a colored Pachner $(2,3)$ move  if 
\begin{enumerate}
\item  the colored diagram $Z$ and $Z'$ are related by a colored Pachner $(2,3)$ move, and
\item  $\psi=\psi'$ on the exteriors $\mathcal{C}(Z)\setminus W=\mathcal{C}(Z')\setminus W'$,
where  $W$ (resp. $W'$) is  the subcomplexes of $\mathcal{C}(Z)$ (resp.  $\mathcal{C}(Z')$)  consisting of the three (resp. two) tetrahedra 
corresponding to the three (resp. two)  crossings of $Z$ (resp. $Z'$) involved in the colored Pachner $(2,3)$ move.
\end{enumerate}
We define other colored moves  on colored singular triangulations  similarly. 

Then the colored Pachner $(2,3)$ move on $\mathcal{CT}$ turns out to be the Pachner $(2,3)$ move on singular triangulations, defined  in Figure \ref{fig:Pach} (a), replacing two tetrahedra sharing one face with three tetrahedra, or its inverse.
See Figure \ref{fig:sandp} for an example, where we color over-strands red and under-strands blue so that we can distinguish them in $3$-spaces in the lowest picture. 

The colored $(0,2)$ move on $\mathcal{CT}$  turns out to be the $(0,2)$ move on singular triangulations, defined  in Figure \ref{fig:Pach} (b),  replacing two adjacent $2$-faces with two tetrahedra, or its inverse.

\begin{figure}
\includegraphics[width=10cm,clip]{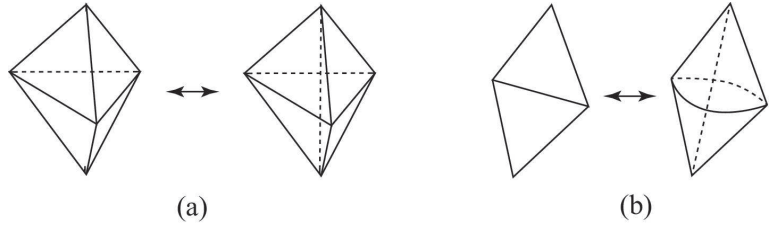}

\caption{(a) The Pachner $(2,3)$ move, (b) The $(0,2)$ move }\label{fig:Pach}
\end{figure}

\begin{figure}
\includegraphics[width=11cm,clip]{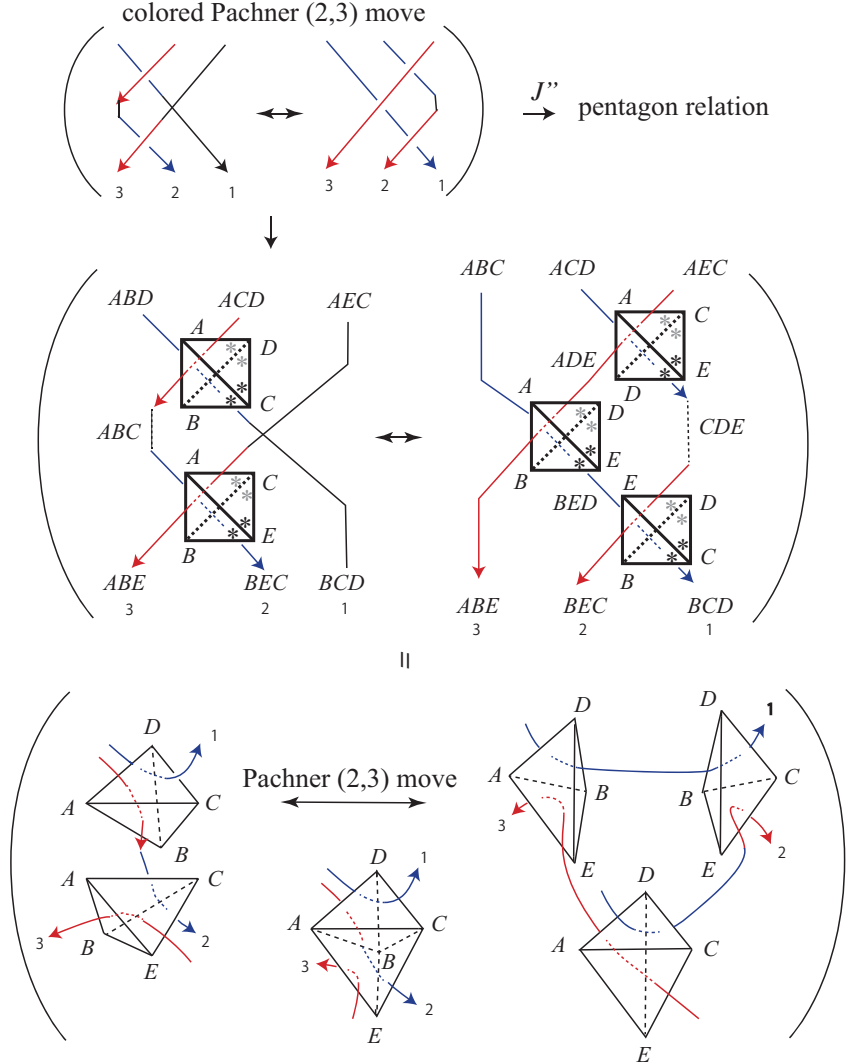}
\caption{A colored Pachner $(2,3)$ move of colored diagrams and a Pachner $(2,3)$ move of colored cell complexes, whose image of $J'$ turns out to be the pentagon relation $S_{23}S_{12}=S_{12}S_{13}S_{23}$ }\label{fig:sandp}
\end{figure}

Correspondingly to  the equivalence relations $\sim_c$  and $\sim'_c$ on $\mathcal{CD}$,
we define  the equivalence relations $\sim_{ct}$ and  $\sim'_{ct}$  on  $\mathcal{CT}$, i.e.,  $\sim_{ct}$ is  generated by all colored moves,
and  $\sim'_{ct}$ is generated by colored moves except for the moves in Figure \ref{fig:exception}.

Let $\pi\co \mathcal{CT} \to \mathcal{CD}$
be the map such that $\pi(X)=Z$ for $X\in \mathcal{CT}(Z)$.

For $\mu=(\mu_1,\ldots,\mu_m), \nu=(\nu_1,\ldots,\nu_m) \in \{\pm\}^m$, $m\geq 0$, recall from  Sect \ref{coloredd}
the subset $\mathcal{CD}(\mu; \nu)\subset \mathcal{CD}$. 
Let  $\mathcal{CT}(\mu; \nu)$ be the set of colored singular triangulations of types  in $\mathcal{CD}(\mu; \nu)$.
Note that $\mathcal{CT}=\bigcup_{\mu, \nu \in \{\pm\}^m, m\geq 0} \mathcal{CT}(\mu; \nu)$.

\begin{proposition}
The composition 
\begin{align*}
J'\circ \pi\co  \mathcal{CT}(\mu; \nu) \to    \bigotimes_{i\in I_+} H(A)_i^{\nu_i} \bigotimes_{j\in I_-} (H(A)^{\op})_j^{\nu_j}
\end{align*}
of the restriction of $\pi$ to $ \mathcal{CT}(\mu; \nu)$ and the universal quantum invariant $J'$ is an invariant   under $\sim'_{ct}$.
If $\gamma^2=1$, then $J'\circ \pi$ is also an invariant under $\sim_{ct}$.
\end{proposition}
\begin{proof}
Note that the projection map $\pi$ induces the map
\begin{align*}
\mathcal{CT}/\sim'_{ct} \to \mathcal{CD}/\sim'_c\quad  (\text{resp. \ } \mathcal{CT}/\sim_{ct} \to \mathcal{CD}/\sim_c),
\end{align*}
which shows the invariance of  $J'\circ \pi$  under $\sim'_{ct}$ (resp. $\sim_{ct}$ if $\gamma^2=1$).
\end{proof}

We call $J'\circ \pi$ the \textit{universal quantum invariant} of colored singular triangulations.
Note that the invariance of $J'\circ \pi$ under  colored Pachner $(2,3)$ moves are shown by pentagon relations.

\section{Octahedral triangulation of tangle complements }\label{gotc}

In this section we define \textit{ideal triangulations of tangle complements}, and construct examples called the \textit{octahedral triangulations}.
We will show that the octahedral triangulation associated to a tangle diagram $D$ naturally  admits a structure of a colored ideal triangulation of type $\zeta(D)$.

\subsection{Ideal triangulations of tangle complements}\label{idtri}

Let $M$ be a compact manifold of dimension $n\le 3$, possibly with non-empty boundary.
Let $F$ be an $(n-1)$-submanifold of $\partial M$.
Let $F_1,..., F_k$ be the connected components of $F$.
Let $M//F$ denote the topological space obtained from $M$ by collapsing each $F_i$ into a point.
An \textit{ideal triangulation} of the pair $(M,F)$ is defined to be a singular triangulation of $M//F$
 such that each vertex of the singular triangulation is on a point arising from $F$.

Let $D_n=[0,1]^2\setminus (P_1\cup \cdots \cup P_n)$ be a punctured disk, where $P_1,\ldots, P_n$ are  small disks with the centers arranged on the line $[0,1]\times \{1/2\}$ as  in Figure \ref{punc}(a).
We define the \textit{leaves-ideal triangulation} $l_n$ of $D_n$ to be the ideal triangulation of the pair $(D_n, ([0,1]\times \{0,1\})\cup \partial P_1 \cup \cdots \cup \partial P_n)$ as in Figure \ref{punc}(b), where we denote by $-\infty, +\infty, p_1,\ldots, p_n$ the vertices corresponding to $[0,1]\times \{0\},  [0,1]\times \{1\}, \partial P_1, \ldots,  \partial P_n$, respectively.
Here we formally define $l_0$ as a segment having $\{\pm \infty\}$ as its vertices.
In particular we call $l_1$ a \textit{leaf}.

\begin{figure}
\includegraphics[width=7cm,clip]{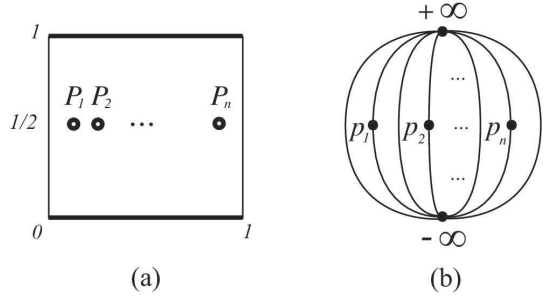}
\caption{(a) A punctured disk and (b) its leaves-ideal triangulation}\label{punc}

\end{figure}

Let $T=T_{1}\cup \cdots \cup T_{n}$ be an  $n$-component tangle.
Let $E=\overline{[0,1]^3 \setminus N(T)}$ be the complement of $T$ in the cube,
where $N(T)$ is a tubular neighborhood of $T$ in the cube.
Let $F_T$ be the intersection $\partial E \cap N(T)$, which consists of  annuli and tori.
Then an \textit{ideal triangulation of the tangle complement} $E$ of $T$ in the cube is defined to be an ideal triangulation of $(E, F_T\cup F_{z=0}\cup F_{z=1})$, where $F_{z=0}=[0,1]\times [0,1]\times \{0\}$ and $F_{z=1}=[0,1]\times [0,1]\times \{1\}$, such that its restriction to each boundary component $[0,1]\times \{0,1\}\times [0,1]$ is a leaves-ideal triangulation. The vertices corresponding to  $F_{z=0}$, and $F_{z=1}$
are denoted by $-\infty$  and $+\infty$, respectively.

\subsection{Colored ideal triangulations for octahedral triangulations of tangle complements}\label{id}

A tangle diagram  $D$ is called \textit{non-splitting} if
\begin{enumerate}
\item the $4$-regular plane graph giving the diagram $D$  is connected,  and
\item there is not a component of $D$ such that crossings  along the path of the component are only over-passing or only under-passing.
\end{enumerate}

Let $T$ be a tangle and $D$ its non-splitting diagram which has at least one crossing.
We  define a cell complex $\mathcal{O}(D)$,  which we call the \textit{octahedral triangulations} associated to $D$, which is an ideal triangulation of the tangle complement $E$.
If in addition $D$ is a link diagram, then $\mathcal{O}(D)$ is nothing but the octahedral triangulation studied in e.g., \cite{CKK14, Yok11} in the context of the hyperbolic geometry.

\noindent\textbf{Step 1. Take a colored diagram}

Recall from Section \ref{j''} the colored diagram $\zeta(D)$  obtained from $D$ by duplicating and thickening the left strands following the orientation.

\noindent\textbf{Step 2. Preparing and placing octahedra}

Let $\{c_{1},\ldots c_{k}\}$ be the set of crossings of the diagram $D$.
In a neighborhood of  $\zeta(c_{i})$, there are four crossings $t_{1}^{i}, t_{2}^{i}, t_{3}^{i}, t_{4}^{i}$ as in Fig \ref{fig:color}, where   $t_{1}^{i}$ is the right crossing when we see strands oriented downwards,  and  $t_{2}^{i}, t_{3}^{i}, t_{4}^{i}$ are defined one by one in a counterclockwise order.
As in Figure  \ref{fig:color}, for $j=1,2,3,4,$ we associate a tetrahedron $\Lambda_{j}^{i}=n^i_j \tilde e^i_{j} \tilde e'^i_{j} s^i_j$ to each $t_{j}^{i}$. 
Then we glue the four tetrahedra  $\Lambda_{1}^{i}, \Lambda_{2}^{i}, \Lambda_{3}^{i}, \Lambda_{4}^{i}$ together  to obtain an octahedron $o_{i}=n^{i}e^{i}_{12}e^{i}_{23}e^{i}_{34}e^{i}_{41}s^{i}$, so that $n^i_j$, $\tilde e^i_{j}$, $\tilde e'^i_{j}$, and $s^i_j$ are going to
$n^i$, $e^i_{j-1,j}$, $e^i_{j,j+1}$, and $s^i$, respectively, where  the index $j$ should be considered modulo $4$. We place $o_{i}$ between the two original strands of $c_i$ so  that $n^{i}$ and $s^{i}$ are placed on the over-strand and the under-strand, respectively.

\begin{figure}
\includegraphics[width=13cm,clip]{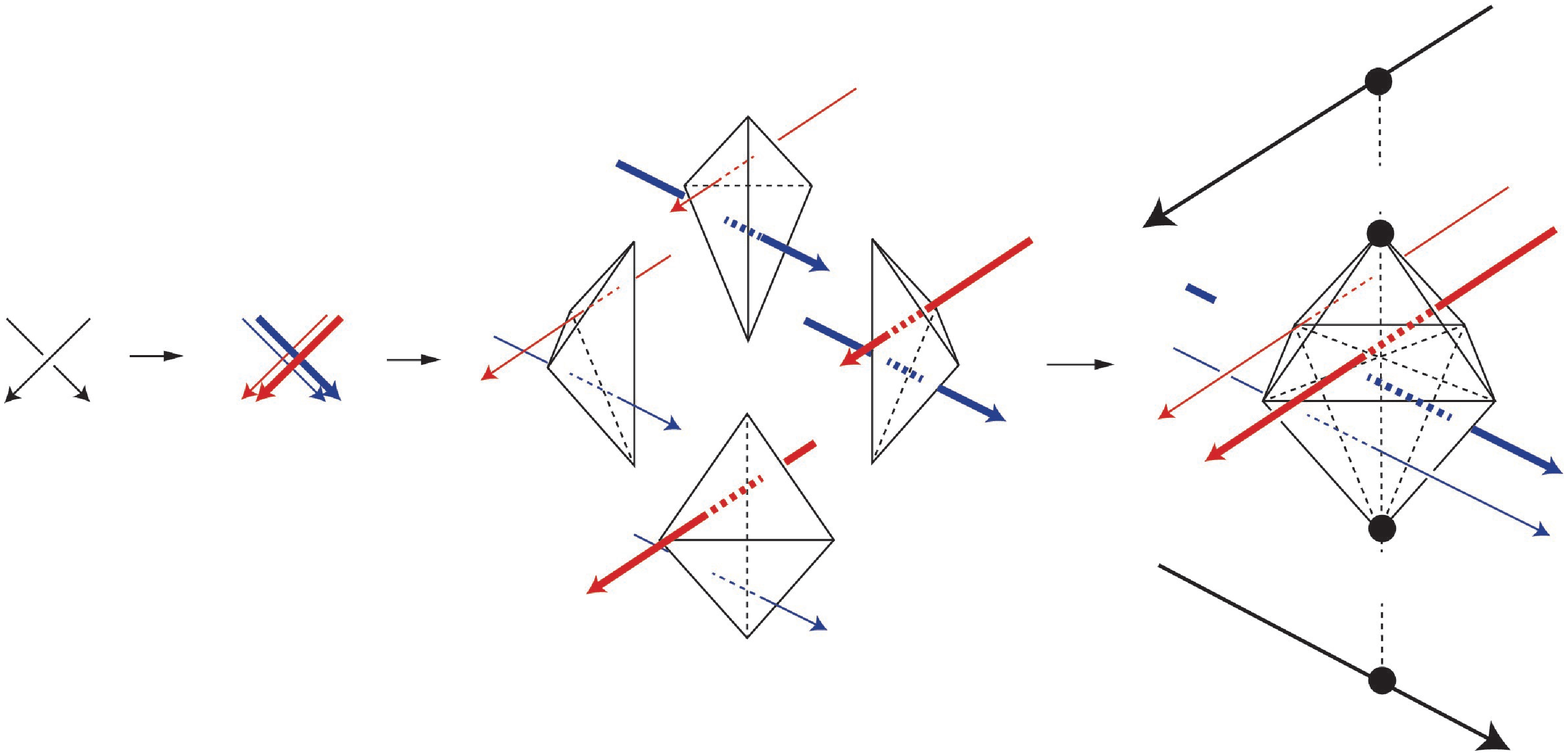}

\vspace{-27pt}
\begin{picture}(0,0)
\put(-172,60){$c_i$}
\put(-100,77){$t^i_1$}
\put(-115,90){$t^i_2$}
\put(-130,77){$t^i_3$}
\put(-115,60){$t^i_4$}

\put(-7,140){$n^i_2$}
\put(-30,127){$\tilde e'^{i}_{2}$}
\put(10,127){$\tilde e^{i}_{2}$}
\put(-7,74){$s^i_2$}

\put(30,105){$n^i_1$}
\put(45,85){$\tilde e'^{i}_{1}$}
\put(50,70){$\tilde e^{i}_{1}$}
\put(30,50){$s^i_1$}

\put(-45,105){$n^i_3$}
\put(-60,88){$\tilde e^{i}_{3}$}
\put(-62,70){$\tilde e'^{i}_{3}$}
\put(-45,50){$s^i_3$}

\put(3,64){$n^i_4$}
\put(-35,40){$\tilde e^{i}_{4}$}
\put(17,40){$\tilde e'^{i}_{4}$}
\put(3,5){$s^i_4$}

\put(143,115){$n^i$}
\put(166,90){$e^{i}_{12}$}
\put(110,90){$e^{i}_{23}$}
\put(103,65){$e^{i}_{34}$}
\put(178,65){$e^{i}_{41}$}
\put(143,25){$s^i$}
\end{picture}

\caption{Octahedral triangulation around a crossing}\label{fig:color}
\end{figure}

\noindent\textbf{Step 3.  Gluing octahedra}
\begin{figure}
\includegraphics[width=3cm,clip]{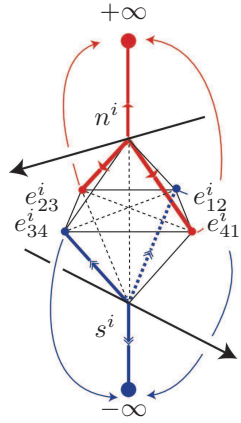}
\begin{picture}(0,0)
\put(-57,110){$n^i$}
\put(-20,80){$e^{i}_{12}$}
\put(-83,80){$e^{i}_{23}$}
\put(-87,70){$e^{i}_{34}$}
\put(-15,70){$e^{i}_{41}$}
\put(-57,30){$s^i$}
\put(-55,150){$+\infty$}
\put(-55,0){$-\infty$}
\end{picture}
\caption{How to glue the edges in a octahedron}\label{fig:attach}
\end{figure}

We glue the octahedra $o_{1},\ldots, o_{k}$
as follows.

For each positive (resp. negative) crossing $c_{i}$, we pull  the vertices $e^{i}_{23}$ and $e^{i}_{41}$ (resp. $e^{i}_{12}$ and $e^{i}_{34}$) upwards, put them on $+\infty$, 
and glue the two edges $n^{i}$-$e^{i}_{23}$ and $n^{i}$-$e_{41}$ (resp. $n^{i}$-$e^{i}_{12}$ and $n^{i}$-$e_{34}$).  Similarly, pull  the vertices $e^{i}_{12}$ and $e^{i}_{34}$  (resp. $e^{i}_{23}$ and $e^{i}_{41}$) downwards, put them on $-\infty$,  and glue the two edges $s^{i}$-$e^{i}_{12}$ and $s^{i}$-$e_{34}$ (resp. $s^{i}$-$e^{i}_{23}$ and $s^{i}$-$e_{41}$), see Figure \ref{fig:attach}. Note that the boundary of the octahedron $o_i$ consists of  four leaves corresponding to the four edge of $c_i$, see Figure \ref{fig:leafs}. 
We glue the  octahedra $o_{1},\ldots, o_{k}$
along the pairs of leaves which are adjacent on  $D$ so that 
$\pm \infty$ are attached compatibly.  
We call the result the \textit{octahedral triangulation} of the complement of $T$ associated to a diagram $D$, and  denote it by $\mathcal{O}(D)$. 
\begin{figure}
\includegraphics[width=3cm,clip]{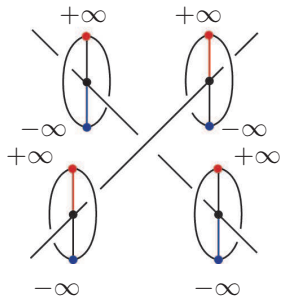}

\vspace{20pt}
\begin{picture}(0,0)
\put(-30,123){$+\infty$}
\put(13,123){$+\infty$}
\put(-50,70){$+\infty$}
\put(35,70){$+\infty$}
\put(-40,20){$-\infty$}
\put(20,20){$-\infty$}
\put(-45,80){$-\infty$}
\put(30,80){$-\infty$}
\end{picture}
\vspace{-20pt}
\caption{Leaves corresponding to the four edges of a crossing}\label{fig:leafs}

\end{figure}

It is not difficult to check that $\mathcal{O}(D)$ is an ideal triangulation of the complement of $T$.
Moreover, we have the following.

\begin{proposition}\label{ci}
The octahedral triangulation $\mathcal{O}(D)$ associated to a tangle diagram $D$ admits a colored ideal triangulation of type $\zeta(D)$.
\end{proposition}
\begin{proof}
Recall that in Step 2 of the definition of $\mathcal{O}(D)$, we associate an octahedron $o_i$ to each crossing $c_i$, where the octahedron is obtained from four tetrahedra as in Figure \ref{fig:color}.  
Actually we can obtain $o_i$ also as the colored cell complex $\mathcal{C}(\zeta(c_i))$ as depicted in Figure \ref{colorcrossing}.
In Step 3, we glued the octahedra and triangles as in Figure \ref{connection},
which follows the gluing rule of the colored tetrahedra and triangles defined in Section \ref{codiag}. 
As the result we have $\mathcal{C}(\zeta(D))$, and finally we identify 
the edges of each octahedron as in Figure \ref{fig:attach},
which gives $\mathcal{O}(D)$ which is singular triangulation of type $\zeta(D)$.
This completes the proof.

\begin{figure}
\includegraphics[width=10cm,clip]{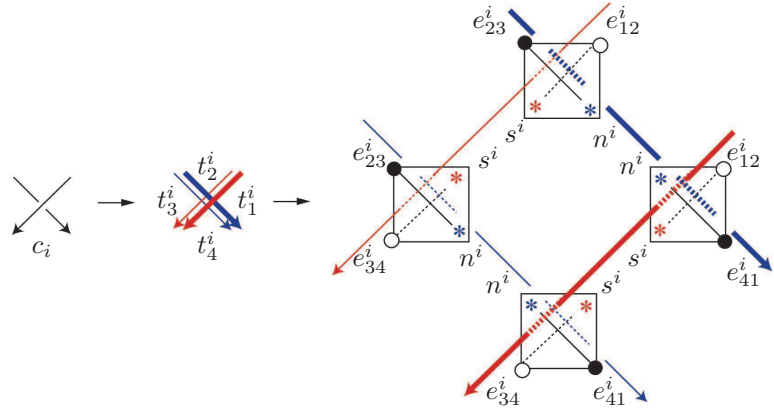}
\caption{The colored ideal triangulation and the octahedron at a crossing of a tangle}\label{colorcrossing}

\begin{picture}(0,0)
\put(-135,100){$c_i$}
\put(-58,117){$t^i_1$}
\put(-73,130){$t^i_2$}
\put(-88,117){$t^i_3$}
\put(-73,100){$t^i_4$}

\put(30,185){$e^{i}_{23}$}
\put(-15,135){$e^{i}_{23}$}
\put(-15,95){$e^{i}_{34}$}
\put(35,45){$e^{i}_{34}$}
\put(80,185){$e^{i}_{12}$}
\put(125,135){$e^{i}_{12}$}
\put(125,90){$e^{i}_{41}$}
\put(75,45){$e^{i}_{41}$}
\put(25,95){$n^i$}
\put(35,85){$n^i$}
\put(75,140){$n^i$}
\put(85,130){$n^i$}
\put(32,132){$s^i$}
\put(42,142){$s^i$}
\put(78,85){$s^i$}
\put(88,95){$s^i$}
\end{picture}

\end{figure}

\begin{figure}
\includegraphics[width=13cm,clip]{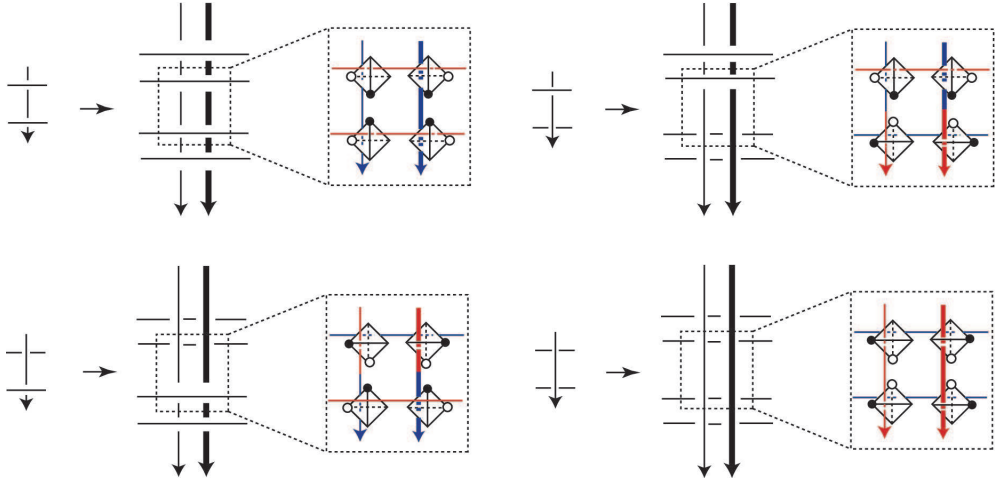}
\caption{How we glued the octahedra in the octahedral triangulations, where the black dots are attached to $+\infty$ and the white dots are attached to $-\infty$.}\label{connection}
\end{figure}

\end{proof}

 
\begin{remark}\label{geom}
A tangle complement could admit more than one colored ideal triangulations up to the equivalence relation $\sim'_{ct}$, and the universal quantum invariant $J'$ could give different values on them. 
We expect that \textit{the universal quantum invariant is an invariant of pairs of $3$-manifolds and some geometrical inputs obtained from the color}, which we  will study in \cite{KST}.
\end{remark}


\end{document}